\documentclass[10pt]{article}

\usepackage{xspace}
\usepackage{url}
\usepackage{mathtools}
\usepackage{amssymb}
\usepackage{amsthm}
\usepackage{empheq}
\usepackage{latexsym}
\usepackage{enumitem}
\usepackage{eurosym}
\usepackage{dsfont}
\usepackage{appendix}
\usepackage{color} 
\usepackage[unicode]{hyperref}
\usepackage{frcursive}
\usepackage[utf8]{inputenc}
\usepackage[T1]{fontenc}
\usepackage{geometry}
\usepackage{multirow}
\usepackage{todonotes}
\usepackage{lmodern}
\usepackage{anyfontsize}
\usepackage{stmaryrd}
\usepackage{natbib}
\usepackage{cleveref}
\usepackage[english=british]{csquotes}
\usepackage[english]{babel}
\usepackage{mathtools}
\usepackage{accents}
\usepackage{color}
\usepackage{tcolorbox}
\usepackage{comment}
\usepackage{appendix}
\usepackage{subfig}

\usepackage{extarrows}

\RequirePackage[subfigure]{tocloft}
\advance\cftbeforesecskip 0.5em\relax
\usepackage[nottoc]{tocbibind}

\setcitestyle{numbers,open={[},close={]}}

\definecolor{red}{rgb}{0.7,0.15,0.15}
\definecolor{green}{rgb}{0,0.5,0}
\definecolor{blue}{rgb}{0,0,0.7}
\hypersetup{colorlinks, linkcolor={red}, citecolor={green}, urlcolor={blue}}
			
\makeatletter \@addtoreset{equation}{section}

\allowdisplaybreaks

\newtheorem{theorem}{Theorem}[section]
\newtheorem{assumption}[theorem]{Assumption}

\newtheorem{lemma}[theorem]{Lemma}

\newtheorem{definition}[theorem]{Definition}
\newtheorem{remark}[theorem]{Remark}

\newtheorem*{definition*}{Definition}

\DeclareUnicodeCharacter{014D}{\=o}
\newcommand{\diff}{\,\mathrm{d}}

\def\balpha{\mathbf{\alpha}}

\newcommand\cA{\mathcal A}

\newcommand\cF{\mathcal F}

\newcommand\cL{\mathcal L}
\newcommand\cM{\mathcal M}

\newcommand\cP{\mathcal P}

\newcommand\cW{\mathcal W}

\newcommand\EE{\mathbb E}
\newcommand\FF{\mathbb F}

\newcommand\PP{\mathbb P}
\newcommand\NN{\mathbb N}
\newcommand\RR{\mathbb R}

\newcommand\XX{\mathbb X}

\def \E{\mathbb{E}}
\def \F{\mathbb{F}}
\def \G{\mathbb{G}}
\def \H{\mathbb{H}}

\def \L{\mathbb{L}}

\def \N{\mathbb{N}}

\def \P{\mathbb{P}}
\def \Q{\mathbb{Q}}
\def \R{\mathbb{R}}
\def \S{\mathbb{S}}

\def \W{\mathbb{W}}
\def \X{\mathbb{X}}
\def \Y{\mathbb{Y}}
\def \Z{\mathbb{Z}}

\def\Ac{\mathcal{A}}
\def\Bc{\mathcal{B}}
\def\Cc{\mathcal{C}}

\def\Ec{\mathcal{E}}
\def\Fc{\mathcal{F}}
\def\Gc{\mathcal{G}}
\def\Hc{\mathcal{H}}

\def\Lc{\mathcal{L}}
\def\Mc{\mathcal{M}}
\def\Nc{\mathcal{N}}

\def\Pc{\mathcal{P}}
\def\Qc{\mathcal{Q}}

\def\Vc{\mathcal{V}}
\def\Wc{\mathcal{W}}

\newcommand{\x}{\mathbf{x}}

\newcommand{\e}{\mathrm{e}}

\newcommand{\xdim}{m}
\newcommand{\bmdim}{d}
\newcommand{\np}{N}

\newcommand{\mybar}[1]{\makebox[0pt]{$\phantom{#1}\overline{\phantom{#1}}$}#1}

\newcommand{\muin}{\mu_{\mathrm{in}}}
\newcommand{\mufin}{\mu_{\mathrm{fin}}}

\newcommand{\camilo}[1]{{\color{magenta}#1}}
\newcommand{\tbd}[1]{{\color{green}#1}}

\def\eps{\varepsilon}
\def\d{\mathrm{d}}

\DeclareMathOperator*{\argmin}{arg\,min}

\newcommand{\msim}{\overset{c}{\sim}}

\begin{document}
\title{Propagation of chaos for mean field Schrödinger problems\footnote{The authors would like to thank Julio Backhoff-Veraguas for helpful discussions. Camilo Hern\'andez acknowledges the support of a Presidential Postdoctoral Fellowship at Princeton University and the Cecilia Tanner Research Fund at Imperial College London. Ludovic Tangpi acknowledges partial support by the NSF under Grants DMS-2005832 and CAREER DMS-2143861.}}

\author{Camilo {\sc Hern\'andez} \footnote{Princeton University, ORFE department, USA. camilohernandez@princeton.edu.} \and Ludovic {\sc Tangpi} \footnote{Princeton University, ORFE department, USA. ludovic.tangpi@princeton.edu}}

\date{\today}

\maketitle

\abstract{
	 In this work, we study the mean field Schr\"odinger problem from a purely probabilistic point of view by exploiting its connection to stochastic control theory for McKean-Vlasov diffusions. 
	Our motivation is to study scenarii in which mean field Schr\"odinger problems arise as the limit of ``standard'' Schr\"odinger problems over \emph{interacting} particles. 
	Due to the stochastic maximum principle and a suitable penalization procedure, the result follows as a consequence of novel (quantitative) propagation of chaos results for forward-backward particle systems. 
	The approach described in the paper seems flexible enough to address other questions in the theory. 
	For instance, our stochastic control technique further allows us to solve the mean field Schr\"odinger problem and characterize its solution, the mean field Schr\"odinger bridge, by a forward-backward planning equation. 
}
\setlength{\parindent}{0pt}

\tableofcontents


\section{Introduction}\label{sec:intro}

In 1932 \citeauthor*{schrodinger1931uber} \cite{schrodinger1931uber,schrodinger1932sur} introduced the problem of determining the most likely path of a cloud of independent scattering particles, conditionally on the observation of their initial, $t=0$, and terminal, $t=T\in (0,\infty)$ distributions. 
This was translated in modern probabilistic terms, see \citeauthor*{follmer1988random} \cite{follmer1988random}, as an entropy minimization problem with marginal constraints by employing the theory of large deviations. 
Let us present this heuristically. 
For positive integers $d$, $i\geq 1$, {\color{black} $\mu_{\rm in}$ and $\mu_{\rm fin}$ probability measures on $\R^d$}, let $\Omega^i:=\Cc([0,T],\R^d)$, $X^i$ be the canonical process, $\P_i$ be the Wiener measures under which the particle $X^i$ has initial distribution $\mu_{\rm in}$, and let $\Omega$ (resp. $\P$) be the associated infinite product space (resp. probability measure). 
Because $(X^i)_{i\ge1}$ are \emph{independent and identically distributed}, it then follows from Sanov's theorem that for any probability measure $\Q$ on $\Cc([0,T],\R^d)$, 
\[
\P\bigg( \frac{1}N \sum_{i=1}^N \delta_{X^i} =\Q\bigg)\underset{N\to \infty}{\approx} \exp\big(-N\cdot   \Hc(\Q|\P_1)\big),\; \text{where } \Hc(\Q| \P_1):=\begin{cases} \displaystyle{\color{black} \E^{\P_{ 1}}}\bigg[ \frac{\d \Q}{\d \P_1}\log \frac{\d \Q}{\d \P_1}\bigg] & \text{if }\Q\ll\P_1\\
\infty&  \text{otherwise}
\end{cases}
\]
denotes the relative entropy between $\P_1$ and $\Q$. 
	That is, the most likely behaviour of the process $(X^i)_{i\geq 1}$ is given by the measure $\Q$ which minimizes the relative entropy $\Hc(\cdot | \P_1)$ with respect to the measure $\P_1$, i.e.,
\begin{equation}
\label{eq:dyn.SP}
\Vc_{\rm S}(\muin,\mufin):= \inf \Big\{ \Hc(\Q|\P_1): \Q\in {\color{black} \Pc(\Omega^1)},\, \Q\circ (X_0^1)^{-1}=\muin, \,\Q\circ (X_T^1)^{-1}=\mufin \Big\}.
\end{equation}
Problem \eqref{eq:dyn.SP} is known as the (dynamic) Schr\"odinger problem.
Its value, the entropic transportation cost, and its optimal solution, the Schrödinger bridge, are respective probabilistic counterparts to the transportation cost and displacement interpolation.
This problem plays an important role in optimal transport theory.
Indeed, the (dynamic) Schrödinger's problem is strongly connected to the (dynamic) Monge--Kantorovich optimal transport problem,
see \citeauthor*{mikami2004monge} \cite{mikami2004monge} or \citeauthor*{leonard2012from} \cite{leonard2012from,leonard2016lazy}. 
We refer to \citeauthor*{leonard2014survey} \cite{leonard2014survey} for an extensive (but by now somewhat outdated) survey on the topic, and to the more recent survey by \citeauthor*{chen2021stochastic} \cite{chen2021stochastic} and the lecture notes of \citeauthor*{nutz2022introduction} \cite{nutz2022introduction} for expositions of the theory.
Let us nonetheless mention that the static version of Schrödinger problem, obtained by projecting\footnote{We refer to \cite{leonard2014survey} for a detailed statement on the connection between the static and the dynamic versions of the Schrödinger problem.}~onto the endpoint marginals has received sustained attention in recent years mostly due to the striking benefit of entropy penalization in computational optimal transport, 
see \citeauthor*{cuturi2013sinkhorn} \cite{cuturi2013sinkhorn}
and \citeauthor*{altschuler2017near} \cite{altschuler2017near}.
In this context, but from a theoretical point of view, we also mention \citeauthor*{bernton2022stability} \cite{bernton2021entropic,bernton2022stability} and \citeauthor*{nutz2021entropic} \cite{nutz2021entropic, nutz2022stability} who study the stability and convergence of static Schrödinger bridges and Schrödinger potentials, respectively. 
	\medskip
	
This work will focus on the dynamic Schr\"odinger problem \eqref{eq:dyn.SP}.
It is less studied from the probabilistic standpoint, but has found several practical applications where the so-called interpolating flow is essential. 
These include problems in economics, see \citeauthor*{galichon2015nonlinear} \cite{galichon2015nonlinear}, network routing, see \citeauthor*{chen2017robust} \cite{chen2017robust}, 
  image interpolation, see \citeauthor*{papadakis2014optimal} \cite{papadakis2014optimal} and \citeauthor*{peyre2019computational} \cite[Chapter 7]{peyre2019computational}, and statistics, see \citeauthor*{hamdouche2023generative} \cite{hamdouche2023generative}, among others.\footnote{For completeness we also mention the works \citeauthor*{monteiller2019alleviating} \cite{monteiller2019alleviating}, and \citeauthor*{lavenant2018dynamic} \cite{lavenant2018dynamic} and \citeauthor*{solomon2015convolutional} \cite{solomon2015convolutional}.}\!
We will develop a new stochastic control and purely probabilistic approach to the dynamic Schr\"odinger problem complementing works by \citeauthor*{mikami2008optimal}  \cite{mikami2006duality, mikami2008optimal} and \citeauthor*{chen2016relation} \cite{chen2016relation,chen2017optimal}. 
\medskip

Let us present the problem in which we will be interested: Let $T>0$ and a positive integer $d$ be fixed, and denote by $(\Omega, \mathcal{F},\PP)$ 
and abstract Polish probability space carrying a sequence of independent $\mathbb{R}^d$-valued Brownian motions $(W^i)_{i\geq 1}$.
For every positive integer $N$, let $\mathcal{ F}_0$ be an initial $\sigma$-field independent of $W^1,\dots, W^N$.
We equip $\Omega$ with the filtration $\mathbb{F}^N:=(\mathcal{F}_t^N)_{t\in[0,T]}$, which is the $\PP$-completion of the filtration generated by $W^1, \dots, W^N$ and $\mathcal{F}_0$, and we further denote by $\mathbb{F}^i:=(\mathcal{F}_t^i)_{t\in[0,T]}$ the $\PP$-completion of the filtration generated by $W^i$ and $\mathcal{F}_0$. 
Without further mention, we use the identifications
\begin{equation*}
	W \equiv W^1;  \quad \cF \equiv \cF^1 \quad \text{and} \quad \FF \equiv \FF^1.
\end{equation*}
It is well--known from the work of \citet*{follmer1988random} (see also \citeauthor*{leonard2012girsanov} \cite{leonard2012girsanov}) that \eqref{eq:dyn.SP} admits the reformulation
\begin{equation}
\label{eq.dyn.SP.control}
	\Vc_{\rm S}(\muin,\mufin)=\inf\bigg\{  \E^\P\bigg[\int_0^T \frac1{2}  \|\alpha_t\|^2\d t\bigg]:  \d X_t=\alpha_t\d t+\d W_t, \, \P\circ (X_0)^{-1}= \muin,\, \P\circ (X_T)^{-1}= \mufin\bigg\},
\end{equation}
where the infimum is taken over \emph{control processes} $ \alpha$ which are $\R^d$--valued, $\F$--adapted and square integrable.
This reformulation gives $\Vc_{\rm S}(\muin,\mufin)$ as the value of a (constrained)
stochastic control problem,
an idea that has been very successfully exploited by \cite{mikami2006duality, mikami2008optimal}
to study the asymptotic behaviour of the zero--noise limit 
as well as by \citeauthor*{chiarini2022entropic} \cite{chiarini2022entropic} and \citeauthor*{tan2013optimal} \cite{tan2013optimal} to study the case of particles following Langevin and semimartingale dynamics instead of Brownian motion, respectively.
A crucial feature of this rich theory
is the \emph{independence} of the particles $(X^i)_{i\geq 0}$.
In this work, we study the most likely path of a cloud of \emph{interacting particles} with given initial and final configurations.
That is, given an interaction potential $\Psi$, we consider
the problem
\begin{align}
\notag
	\Vc^N_{\rm e}(\muin,\mufin):=\inf\bigg\{\E^{\P}\bigg[ \frac1{2N} \sum_{i=1}^N \int_0^T\! \|\alpha_t^i\|^2\d t\bigg]: &\;  \d X_t^i=\Big( \alpha_t^i-\frac1N \sum_{j=1}^N \nabla\Psi (X_t^i- X_t^j )\Big) \d t+\d W_t^i,\\\label{eq:N.SP.intro}
&  \; \P\circ (X_0^i)^{-1} =\muin,\; \P\circ ( X_T^i)^{-1} = \mufin  ,\; 1\leq i\leq N\bigg\},
\end{align}
{where $X_0^{1},\dots, X_0^{N}$ are i.i.d.~$\Fc_0$-measurable $\R^\xdim$-valued random variables} and the infimum is taken over \emph{vectors} $(\alpha^1,\dots,\alpha^N)$ such that each $\alpha^i$ is $\R^d$--valued, $\F^N$--adapted and square integrable.
This problem was first studied by \citeauthor*{backhoff2020mean} \cite{backhoff2020mean}.
Because of the high intractability of the problem for $N$ large, 
these authors borrowed inspiration from the theory of McKean--Vlasov control to introduce the so-called \emph{mean field (entropic) Schrödinger problem} 
\begin{align*}
 	\Vc_{\rm e}(\muin,\mufin):=\inf\Big\{\mathcal{H}(\Q |\Gamma(\Q)) :   \d X_t= -(\nabla \Psi \ast {\color{black} \Q_t} ) (X_t)  \d t+\d W_t,\;\Gamma(\Q):={\color{magenta} \Q}\circ( X)^{-1},\, \Q \circ (X_0)^{-1}=\muin, \, \Q \circ (X_T)^{-1}=\mufin  \Big\},
 \end{align*}
 where $\ast$ denotes the convolution operator, and {\color{black}the infimum is taken over probability measures $\Q$ on $\Cc([0,T],\R^d)$ with finite first moment and whose $t$-marginal we denote $\Q_t$}. The value $\Vc_{\rm e}(\muin,\mufin)$ and its optimizers are referred to as the \emph{mean field entropic transportation cost} and \emph{mean field Schr\"odinger bridges} (MFSB for short), respectively. 
In complete analogy to the independent case, \cite{backhoff2020mean} showed that
\begin{align}\label{MFSP:entropic}
\begin{split}
 	\Vc_{\rm e}(\muin,\mufin)=\inf\bigg\{\E^\P\bigg[\int_0^T  \frac1{2}  \|\alpha_t\|^2 \diff t\bigg]:&  \; \d X_t =\big( \alpha_t - (\nabla \Psi \ast \P\circ (X_t)^{-1}) (X_t)\big)\d t +\d W_t, \\
	&  \; \P\circ (X_0)^{-1}=\muin,\, \P\circ  (X_T)^{-1}=\mufin  \bigg\},
\end{split}
\end{align}
where the infimum is taken over $\R^d$--valued, $\F$--adapted and square integrable processes $\alpha$.
 In this setting, the optimal states $\hat X$ correspond, via their law, to mean field Schr\"odinger bridges.

\medskip
 
The main motivation of the present work is to give conditions guaranteeing that the mean field Schr\"odinger problem actually arises as the limiting value of the sequence of finite particle Schr\"odinger problems \eqref{eq:N.SP.intro}.
That is, rigorously proving the limit	$\Vc_{\rm e}^N\longrightarrow \Vc_{\rm e}$ as $N$ goes to infinity.
This is strongly connected to the so-called convergence problem in mean field game theory \citeauthor*{cardaliaguet2015master} \cite{cardaliaguet2015master}.
Our approach will be based on stochastic calculus arguments reminiscent of those of \citeauthor{lauriere2021convergence} \cite{lauriere2021convergence} that allow reducing the convergence problem into a ``forward-backward propagation of chaos'' issue.
Due to the singularity of the terminal cost, our analysis begins with a suitable penalization argument that allows us to to remove the constraint on the law of $X_T$, thus transforming the stochastic optimal transport problem into a ``standard'' stochastic control problem of McKean-Vlasov type.
A map $g:\Pc_2(\R^m)\longrightarrow [0,\infty)$ is said to be a \emph{smooth penalty function} if: $(i)$ $g(\mu)=0$ if and only if $\mu=\mufin$; $(ii)$ $g$ is L-differentiable. 
We refer to \Cref{sec:MFSP} for details as well as a reminder of the notion of $L$-differentiability and examples of penalty functions.
Given a smooth penalty function $g$, for each $k\ge 1$, we consider the penalized problem
\begin{equation}
\label{eq:inf.entropic.problem-k.penalized}
	\Vc^{k}_{\rm e}(\muin, \mufin) := \inf \bigg\{ \EE\bigg[\int_0^T\frac12 \|\alpha_t\|^2 \d t + kg(\mathcal{L}(X_T)) \bigg]:    \; \d X_t =\big( \alpha_t - (\nabla \Psi \ast \P\circ (X_t)^{-1}) (X_t)\big)\d t +\d W_t, \; \P\circ (X_0)^{-1}=\muin \bigg\},
\end{equation}

which, by strong duality, allows us to see $\Vc_{\rm e}$ as the limit, as $k$ goes to infinity, of the penalized problems \eqref{eq:inf.entropic.problem-k.penalized}, i.e., $\Vc_{\rm e}^{k}\longrightarrow \Vc_{\rm e}$ as $k$ goes to infity. 
	By introducing the corresponding penalized version $\Vc^{N,k}_{\rm e}$,
	the stochastic maximum principle developed by \citeauthor*{carmona2015forward} \cite{carmona2015forward}
	reduces the convergence problem to a forward-backward propagation of chaos with the caveats that for any $k$, the $N$-particle system associated to $V^{N,k}_{\rm e}$ is of McKean--Vlasov type itself!
	This makes the results in \cite{lauriere2021convergence} 
	inoperable. 
	Moreover, using a synchronous coupling argument along with techniques developed in \cite{lauriere2021convergence} would require imposing conditions on the coefficients of the particle system that depend on $k$, which is not desirable given that $k$ needs to go to infinity eventually.
	One way we propose to get around these issues is to impose a convexity property on the penalty function well-known in optimal transport theory, see \citeauthor*{mccann1997convexity} \cite{mccann1997convexity}, called {\it displacement convexity}, see also \citeauthor*{carmona2018probabilisticI} \cite[Proposition 5.79]{carmona2018probabilisticI} for its equivalent formulation in terms of $L$-differentiability which we adopt here.
	The relevance of this type of convexity in the convergence problem in mean field games was recently highlighted by \citeauthor*{jackson2023quantitative} \cite{jackson2023quantitative}.
	{Since verifying whether a target measure $\mufin$ admits a displacement convex penalty function might be a hard task, we will state our convergence result in terms of a weaker version of the above problem.
	In the weak version, the terminal condition $\P\circ  (X_T)^{-1}=\mufin$ is replaced by $\P\circ  (X_T)^{-1}\leq_{\rm cvx}\mufin$, where $\leq_{\rm cvx}$ denotes the convex ordering of probability measures.
	As we will see, the weak version of the problem is naturally embedded with the convexity required by our approach, see \Cref{lemma.reg.gphi}.
	In particular, there is no need to make extra assumptions on the measure $\mufin$ beyond the feasibility of the problems.}


	\medskip
	


In addition to investigating the large particle limit, we also derive
 new existence and characterization results for mean field Schr\"odinger bridges.
We show that for interaction functions that are nonlinear in the law of the state value, the mean field Schr\"odinger problem admits an optimal Schr\"odinger bridge that, in addition, can be characterize as the solution of a planning FBSDE.
This extends the results first obtained in \cite{backhoff2020mean} using very different arguments. 
While the focus of this work is on the case of \emph{interacting particles}, the approach (and some results) are new even for the classical problem \eqref{eq:dyn.SP}.
For instance, we will consider non--quadratic cost functions as well as general drifts.
Let us illustrate our main findings on the following extension of the entropic Schr\"odinger problem discussed so far.
\begin{align}
\label{eq:problem.ex.intro}
\begin{split}
 \Vc (\muin,\mufin):=\inf\bigg\{\E^\P\bigg[\int_0^T  f_1(t, \alpha_t) \diff t\bigg]:&  \; \d X_t =\big( \alpha_t - (\nabla \Psi \ast \P\circ (X_t)^{-1}) (X_t)\big)\d t +\d W_t, \\
	&  \; \P\circ (X_0)^{-1}=\muin,\, \P\circ  (X_T)^{-1}=\mufin  \bigg\},
\end{split}
\end{align}

\paragraph*{Main results for the at most quadratic Schr\"odinger problem.}

Our first set of results provides the existence of an MFSB and establishes a connection between the above problem and a class of McKean-Vlasov FBSDEs. \medskip

\begin{theorem}
\label{thm:existence.mfsb}
	Suppose $\Vc (\muin,\mufin)<\infty$. Let $\Psi: \RR^\xdim  \longrightarrow \RR$ be twice differentiable with Lipschitz--continuous first order derivative.
	Let $f_1:[0,T]\times \R^\xdim \longrightarrow \R^\xdim$ be continuously differentiable in the second argument, convex, positive, {satisfy $f_1(t,0)=0$}, and suppose there are $p>1$ and $C_1\in\R$, $C_2,\ell_{f_1}>0$, such that \( C_1+C_2\|a\|^p\leq f_1(t,a)\leq \ell_f +\ell_f \|a\|^2\).
	Then, 
	\begin{enumerate}[label=$(\roman*)$, ref=.$(\roman*)$,wide,  labelindent=0pt]
		\item	 the mean field entropic Schr\"odinger problem \eqref{eq:problem.ex.intro} admits a mean field Schr\"odinger bridge $\hat X$ with associated optimal control $\hat \alpha$;
		\item for $\Lambda(t,y)\in \argmin_{a\in \R^\xdim} \big \{ f_1(t,a)+a\cdot y\big\}$, the {\rm MFSB} $\hat X$ and its associated optimal control $\hat \alpha$ satisfy
	\[ \hat X \stackrel{\text{d}}{=} \mybar X, \text{ and, } {\color{black} \hat\alpha_t \stackrel{\text{d}}{=} \Lambda(t,\mybar Y_t)},\text{ for almost all } t\in [0,T],\]
	for $( \mybar X, \mybar  Y,  \mybar Z)$, defined on a probability space $(\mybar \Omega, \mybar \Fc, \mybar \P)$ supporting a Brownian motion $\mybar W$, satisfying the forward backward stochastic differential equation
	\begin{equation}\label{eq:fbsde.MFSP.charact}
		\begin{cases}
			\diff \mybar X_t = \big( \Lambda(t, \mybar Y_t) - \tilde \E [  \nabla \Psi (\mybar X_t-\tilde X_t)]\big)   \diff t + \diff 
			 \mybar W_t,\,  t\in [0,T]\\
			\diff \mybar Y_t = -\big( \tilde \E  [ \nabla^2 \Psi( \mybar X_t-\tilde X_t)\cdot \mybar Y_t]-\tilde \E [\nabla^2 \Psi (\tilde X_t-\mybar X_t)\cdot \tilde Y_t ]\big)  \diff t + \mybar Z_t\diff 
			 \mybar W_t ,\,  t\in [0,T)\\
			\mybar X_0\sim \muin,\quad \mybar X_T \sim \mufin,
		\end{cases}
	\end{equation}
	where $\tilde X$ and $\tilde Y$ denote independent copies of $ \mybar X$ and $\mybar Y$, respectively, on some probability space $(\tilde \Omega, \tilde \Fc, \tilde \P)$ and $\tilde \E$ denotes the associated expectation operator. In particular,
\[
\mybar M_t:= \mybar Y_t -\int_0^t\big( \tilde \E [  \nabla^2 \Psi ( \mybar X_r-\tilde X_r) \cdot \mybar Y_r ]-\tilde \E [\nabla^2 \Psi (\tilde X_r-\mybar X_r)  \cdot \tilde Y_r] \big) \diff r,
\]
is a {$(\mybar \P,\mybar \F^{\bar \P})$-martingale on $[0,T)$, where $\mybar\F^{\bar \P}$ denotes the $\mybar\P$-augmented filtration generated by $\mybar W$.}

		\item \label{thm:existence.mfsb.3} if $f_1\equiv \|a\|^2/2$ and $\Psi$ is symmetric,\footnote{ This is, $\Psi(x)=\Psi(-x), x\in \R^\xdim$.} the {\rm MFSB} $\hat X$ and its associated optimal control $\hat \alpha$ satisfy 
		$$\hat X \stackrel{\text{d}}{=} \mybar X,\; \hat\alpha_t \stackrel{\text{d}}{=} \mybar Y_t, \text{ for almost every } t\in [0,T],$$ 
		for $(\mybar X, \mybar Y, \mybar Z)$, defined on a probability space $(\mybar\Omega,\mybar \cF, \mybar \P)$ supporting a Brownian motion $\mybar W$, satisfying 
	\begin{equation*}
		\begin{cases}
			\diff \mybar X_t = \big(  \mybar Y_t - \tilde \E [  \nabla \Psi (\mybar X_t-\tilde X_t)]\big)   \diff t + \diff 
			 \mybar W_t,\, \textcolor{magenta}{t\in [0,T]}\\
			\diff \mybar Y_t =  \tilde \E [ \nabla^2 \Psi (\mybar X_t-\tilde X_t)\cdot (\mybar Y_t-\tilde Y_t) ] \diff t   + \mybar Z_t\diff 
			 \mybar W_t ,\, \camilo{ t\in [0,T)}\\
			\mybar X_0\sim \muin,\quad \mybar X_T \sim \mufin,
		\end{cases}
	\end{equation*}
	where, as in $(ii)$, $\tilde X$ and $\tilde Y$ denote independent copies of $ \mybar X$ and $\mybar Y$. In particular, $\mybar M$ is given by
	\[
	M_t:=  \mybar Y_t -\int_0^t  \tilde \E [ \nabla^2 \Psi (\mybar X_t-\tilde X_t)\cdot (\mybar Y_t-\tilde Y_t) ]   \diff r.
	\]
\end{enumerate}
\end{theorem}
\Cref{thm:existence.mfsb} ensures the existence of an MFSB for \eqref{eq:problem.ex.intro}. 
	In addition, its associated optimal control is given in terms of a solution to the FBSDE system \eqref{eq:fbsde.MFSP.charact}. 
	Recalling that in the case of non-interacting particles the optimal control is a martingale, see \citeauthor*{lehec2013representation} \cite[Lemma 11]{lehec2013representation}, we complement our results by identifying the form of the martingale process associated with the optimal control in the case the particles interact through the potential $\nabla \Psi$.
	{We remark that under assumptions of $\muin$ and $\mufin$ that guarantee feasibility, i.e., $\Vc_{\rm e}(\muin,\mufin)<\infty$, {\rm \Cref{thm:existence.mfsb}\ref{thm:existence.mfsb.3}}, which specializes the results to the entropic case \eqref{MFSP:entropic}, had already appeared in \cite[Theorem 1.3]{backhoff2020mean}, see the Assumption ${\rm (H2)}$ and Equation $(4)$ therein.} 
	Observe that, similar to \cite[Theorem 1.3]{backhoff2020mean}, the second statement above {\color{black}provides a characterization of an optimizer.} 
	However, our results go well beyond the entropic problem. 
	In fact, our analysis in \Cref{sec:MFSP} generalizes \Cref{thm:existence.mfsb} to a class of Schr\"odinger problems beyond \eqref{eq:problem.ex.intro}, see \Cref{thm:existence.mfsb.body.paper} for details. 
	In particular, we show that the MFSB and the associated optimal control can be derived as limits (in law) of a sequence of solutions of McKean--Vlasov equations derived from \eqref{eq:inf.entropic.problem-k.penalized}.
	Said result holds for a fairly general class of drifts and for cost functionals that are not necessarily quadratic in the control process and are allowed to depend on the state and law of the controlled process.
	{\color{black}Even so, \Cref{thm:existence.mfsb.body.paper} is interesting not only because of its generality but also because it ellucidates a legitimate method to compute MFSBs by means of McKean--Vlasov equations, objects for which there are known theoretical guarantees for numerical implementations, see for instance the recent results in \citeauthor*{baudelet2023deep} \cite{baudelet2023deep}, \citeauthor*{han2022learning} \cite{han2022learning}, as well as the review of \citeauthor*{lauriere2021numerical} \cite{lauriere2021numerical}.
		What's more, the same reasoning reveals a procedure to numerically approximate solutions to \eqref{eq:fbsde.MFSP.charact}.
		This should be contrasted with the naïve approach of trying to solve \eqref{eq:fbsde.MFSP.charact} directly for which, to the best of our knowledge, there is no theoretical analysis or numerical implementations at this level of generality.} 


\medskip

Our second result is about the convergence of a \emph{weaker} version of the finite particle problem \eqref{eq:N.SP.intro}, to its mean field counterpart obtained from \eqref{MFSP:entropic}.
	By placing the condition $\P\circ  (X_T)^{-1}\leq_{\rm cvx}\mufin$ on the terminal laws instead of $\P\circ  (X_T)^{-1}=\mufin$, we introduce and establish the convergence of $\Vc_c^N(\muin,\mufin)$ to its mean field counterpart $\Vc_c(\muin,\mufin)$.

\begin{theorem}
\label{thm:convergence.mfsb}
Suppose $\Vc_c(\muin,\mufin), \Vc_c^N(\muin,\mufin)<\infty$, $N\geq 1$. Let $a\longmapsto f_1(t,a)$ be uniformly strongly convex and $\Psi: \RR^\xdim  \longrightarrow \RR$, $\Psi(x)=\|x\|^2/2$.
	Then, we have that
	\begin{equation*}
		\lim_{N\to \infty}\Vc^N_c(\muin,\mufin) = \Vc_c (\muin,\mufin).
	\end{equation*}
\end{theorem}

As alluded to above, our arguments reduce the convergence problem to a type of forward-backward propagation of chaos result which does not follow from standard results since the terminal condition of the $N$-particle system depends on the law of the forward process $X$ due to constraint on the law of $X_T$. 
	\Cref{thm:convergence.mfsb} presents one scenario under which the finite particle approximation holds. 
	{The result holds for arbitrary finite time horizons and, by virtue of working with the weaker problems, does not require any other additional assumption beyond the feasibility of the problems.
	Back in the problems $V_{\rm e}(\muin,\mufin)$ and $V^N_{\rm e}(\muin,\mufin)$, the convexity inherent to the weaker problem is in general absent and the finite particle approximation holds under an assumption on the target measure, see {\rm \Cref{rmk:convergence}}.} 
	As with our characterization, the analysis in \Cref{sec:finiteSPapprox} extends \Cref{thm:convergence.mfsb} to a class of Schrödinger problems beyond \eqref{MFSP:entropic}, see \Cref{thm.prop.chao.c} for details.
The reader familiar with mean field problems might recognize the role of the previous convex condition but wonder whether there are other known scenarii in which the convergence is possible. 
Indeed, in the study of the convergence problem in mean field games (and uniqueness of mean field games), either small time conditions or monotonicity/convexity conditions are always needed to achieve strong convergence results \cite{cardaliaguet2015master,jackson2023quantitative,lauriere2021convergence}.
In our setting, the FBSDE associated with $V^{N,k}_{\rm e}$ is also well-behaved over a small time horizon, said $T_o^k$. 
Yet, the value of $T_o^k$ depends on the Lagrange multiplier $k$ (appearing in the terminal condition), which makes this approach unviable since, as explained above, $k$ needs to go to infinity eventually. 
	 We elaborate more on this in \Cref{rmk.convergenceN.regularity}.

\paragraph*{Organization of the paper.}
The rest of this paper is organized as follows: we dedicate the rest of this section to fixing some frequently used notation. 
	\Cref{sec:MFSP} begins presenting the formulation of the mean field Schr\"odinger problem and the standing assumptions on the data of the problem. It also introduces the penalization argument. \Cref{sec:mfspexistence} provides the proofs of the existence result as well as the characterization in the general setting, see \Cref{thm:existence.mfsb.body.paper}, which we then specialize to the case of \Cref{thm:existence.mfsb}. 
	\Cref{sec:finiteSPapprox} introduces the finite particle Schr\"odinger problem and presents the strategy for the analysis. \Cref{sec.approx.disp.conv} presents the proof of the finite particle approximation for the weaker problems, \Cref{thm.prop.chao.c}. \Cref{sec.proof.entropic.conv} brings the analysis back to the at-most quadratic problem presenting the proof of \Cref{thm:convergence.mfsb}.
	An Appendix section includes some technical auxiliary results. 

\paragraph*{Notation.} Let $p$ be a positive integer.
	For $(\alpha,\beta) \in \R^p\times\R^p$, $\alpha\cdot \beta$ denotes their inner product, with associated norm $\|\cdot\|$, which we simplify to $|\cdot|$ when $p$ is equal to $1$. 
	$I_p$ denotes the identity matrix of $\R^{p\times p}$.
 	{Given $f:[0,T]\times\R^p \longrightarrow \R$ continuously differentiable in the second argument, 
	we say that the mapping $x\longmapsto f(t,x)$ is uniformly Lipschitz smooth (with constant $\ell_{f}>0$) if $|\partial_x f(t,x)-\partial_x f(t,y) |\leq \ell_{f} \|x-y\|$, for all $t\in [0,T]$, $x,y\in \R^p$.
	The mapping $x\longmapsto f(t,x)$ is uniformly strongly convex (with constant $\lambda_{f}>0$) if $(\partial_x f(t,x)-\partial_x f(t,y))\cdot (x-y)\cdot  \geq \lambda_{f} \|x-y\|^2$, for all $t\in [0,T]$, $x,y\in \R^p$.
	Moreover, we denote by $(f_1(t,\cdot))^\star (y)$ its convex conjugate in the second coordinate, i.e., $(f_1(t,\cdot))^\star (y)=\sup_{x\in \R^p} \{ x\cdot y -f(x)\}$.}\medskip
	
	Let $E$ be an arbitrary Polish space with a metric $d_E$. 
	Given a positive integer $\ell$ and a vector $x \in E^\ell$, we will denote by
\[
L^\ell(x):=\frac1{\ell}\sum_{j=1}^{\ell}\delta_{x^j},
\]
the empirical measure associated with $x$. 
	We denote by $\Pc(E)$ the set of probability measures on $E$ (endowed with $\Bc(E)$, the Borel $\sigma$-algebra for the topology induced by $d_E$) and by $\Pc_p(E)$ the subset of $\Pc(E)$ containing measures with finite $p$-th moment. 
	Notice then that for any $x\in E^\ell$, we have $L^\ell(x)\in\Pc_p(E)$. 
	{For $\nu\in\Pc(E)$, and $\varphi:E\longrightarrow \R$, we set $ \mu( \varphi):= \int_{E} \varphi (e)\mu(\d e)$}.
	On a product space $E_1\times E_2$, $\pi^i$ denotes the projection on the $i$-th coordinate and for $\nu\in \Pc(E_1\times E_2)$, $\pi^i_{\#}\nu$ denotes the pushforward measure, i.e. $\pi^i_{\#}\nu (A):=\nu\circ (\pi^i)^{-1}(A)\in \Pc(E_i)$, for $A\in E_i$.
	$\Pc_p(E)$ is always endowed with the distance $\cW_{p}$ define next. 
	For $\mu, \nu\in \Pc_p(E)$, we denote by $\cW_{p}(\mu, \nu)$ the $p$-Wasserstein distance between $\mu$ and $\nu$, that is, letting $\Gamma(\mu,\nu)$ be the set of all couplings $\pi$ of $\mu$ and $\nu$, i.e., probability measures on $(E^2,\Bc(E)^{\otimes 2})$ with first and second marginals $\mu$ and $\nu$, respectively, we have that
\begin{equation*}
	\cW_{p}(\mu, \nu) := \bigg(\inf_{\pi\in\Gamma(\mu,\nu)}\int_{E\times E}d_E(x,y)^p\pi(\mathrm{d}x,\mathrm{d}y) \bigg)^{1/p}.
\end{equation*}
For a random variable $X$ on $(\Omega,\Fc,\P)$, we write $X\sim \mu$ whenever $\Lc(X)=\mu$, where $\cL(X)$ denotes the law of a random variable $X$ with respect to $\PP$. Given random variables $X, Y$, we write $X\stackrel{d}{=}Y$ whenever they are equal in distribution.

\section{Problem statement and penalization}\label{sec:MFSP}
As already mentioned in the introduction, thanks to the flexibility afforded by our method, we will actually prove more general results than announced.
	In particular, we will allow the interaction potential and cost to be generic, possibly nonlinear functions of the measure argument.
	Let us thus present the general class of problems under consideration.\medskip

Fix a nonnegative integer $m$ and a function $f_1:[0,T]\times A \longrightarrow \R$, we let $\mathfrak{A}$ be the set defined as 
\begin{equation*}
	\mathfrak{A} := \bigg\{\alpha:[0,T]\times \Omega \longrightarrow A,\,\, \FF\text{--progressive processes such that } \EE\bigg[\int_0^Tf_1(t,\alpha_t) \diff t\bigg]<\infty\bigg\},
\end{equation*}
with $A$ a given nonempty, convex, and closed subset of $\mathbb{R}^m$.\footnote{In the typical example, $A=\R^\xdim$.}
 Given two probability distributions $\muin, \mufin \in \mathcal{P}_2(\mathbb{R}^m)$, we are interested in the general stochastic optimal transport problem
\begin{equation}
\label{eq:problem}
	V(\muin, \mufin) := \inf\bigg\{\EE\bigg[\int_0^Tf(t, X^\alpha_t, \alpha_t,\mathcal{L}(X^\alpha_t))\diff t \bigg]: \alpha \in \mathfrak{A},\,  X_0^\alpha\sim\muin,\, X_T^\alpha \sim   \mufin \bigg\},
\end{equation}
where $X^\alpha$ is the unique strong solution of the controlled stochastic differential equation
\begin{equation}
\label{eq:control.SDE.MF}
	X^\alpha_t = X_0^\alpha + \int_0^t\alpha_u+b(u, X^\alpha_u, \mathcal{L}(X^\alpha_u))\diff u + \sigma W_t,\; t\in [0,T],
\end{equation}
where $\sigma\in  \R^{\xdim\times d}$. The assumptions on the functions $b$, $f_1$ and $ f$ as well as the matrix $\sigma$ will be introduced below.

\begin{definition}
	Given $\muin, \mufin \in \cP_2(\RR^\xdim)$ we call $V(\muin,\mufin)$, with some abuse of language, the value of the mean field Schr\"odinger problem and when the cost $f$ is given by $f\equiv \|a\|^2/2$, then $V(\muin,\mufin)$ is called the mean field entropic transportation cost.
	Moreover, the optimal states of $V(\muin,\mufin)$ are mean field Schr\"odinger bridges.
\end{definition} 	

Problem \eqref{eq:problem} means that we are looking for the ``best trajectory'', in the sense of minimal cost $f$, to transport the probability distribution $\muin$ to $\mufin$. {Throughout the paper, we make the assumption $V(\muin, \mufin)<\infty$, i.e., \eqref{eq:problem} is feasible, and refer to \Cref{lemma.feasibility} and \cite{backhoff2020mean} for sufficient conditions on $b$ and $\muin$ and $\mufin$ guaranteeing it.} \medskip

\begin{remark}\label{remark:examples}
We make a few comments on the nature of {\rm Problem \eqref{eq:problem}}$:$
\begin{enumerate}[label=$(\roman*)$, ref=.$(\roman*)$,wide,  labelindent=0pt]
	\item In light of the presence of the terminal condition, we may interpret $V(\muin,\mufin)$ as a planning McKean--Vlasov stochastic control problem, i.e., the cooperative counterpart of the mean field game planning problem introduced in {\rm \citeauthor*{lions2007theorie} \cite{lions2007theorie}}. \label{rmk.examples.i}
	\item When $f(t, x, a,\mu)= \|a\|^2/2$, and, $b(t, x, \mu) = -\int_{\R^m}\nabla\Psi(x-\tilde x)\mu(\d \tilde x)$, then Problem \eqref{eq:problem} becomes the mean field entropic Schr\"odinger problem \eqref{MFSP:entropic}. 
	If in addition $b=0$, it becomes Problem \eqref{eq.dyn.SP.control}.
These two settings will serve as our standard examples.\label{rmk.examples.iii}
\end{enumerate}
\end{remark}

We now present the standing assumptions used in this article. 
	To that end, let us recall the notion of $L$-derivative on the space of probability measures, see for instance \citeauthor*{lions2007theorie} \cite{lions2007theorie} or~\citeauthor*{carmona2018probabilisticI} \cite[Chapter 5]{carmona2018probabilisticI} for further details. 
	The definition is based on the notion of the lifting of a map $\phi:\mathcal{P}_2(\mathbb{R}^m)\longrightarrow \mathbb{R}$ to functions $\Phi$ defined on the Hilbert space $\L^2(\tilde \Omega ,\tilde \Fc ,\tilde \P;\mathbb{R}^m)$, over some atomless probability space $(\tilde \Omega,\tilde \Fc,\tilde \P)$, by setting $\Phi:\L^2(\tilde \Omega,\tilde \Fc,\tilde \P;\mathbb{R}^m)\longrightarrow \mathbb{R}$, $\Phi(\xi)=\phi(\Lc(\xi))$. 


\begin{definition}\label{def.L.diff}
A function $\phi:\mathcal{P}_2(\mathbb{R}^m)\longrightarrow \mathbb{R}$ is said to be continuously $L$-differentiable if the lifted function $\Phi$ is continuously Fréchet differentiable with derivative $\partial_\xi \Phi:\L^2( \Omega,  \Fc, \P;\mathbb{R}^m)\longrightarrow \L^2( \Omega,  \Fc, \P;\mathbb{R}^m)$. 
	In this case, there is a measurable function $\partial_\mu \phi(\mu): \mathbb{R}^m\longrightarrow \mathbb{R}^m$, which we call the L-derivative of $\phi$, which for $\xi, \zeta\in \L^2(  \Omega,  \Fc,  \P;\mathbb{R}^m)$, $\Lc(\xi)=\mu$, satisfies, $\partial_\mu \phi(\mu)(\xi)=\partial_\xi \Phi(\xi)$, and 
\[
\phi(\Lc(\xi+\zeta))-\phi(\mu)= \E[ \partial_\mu \phi(\mu)(\xi)\cdot \zeta]+o(\|\zeta\|_{\L^2}).
\]
\end{definition}

In particular, the derivative is uniquely defined up to an additive constant. 
	To ease the notation, the above definition of $L$-differentiability assumes the lifting is defined on the same probability space in which the mean field Schrödinger problem was defined, but any probability space supporting a random variable with uniform distribution on $[0, 1]$ can be used.


\begin{assumption}
We assume the matrix $\sigma$ is nonzero as well as:
\label{ass.b.f}
	\begin{enumerate}[label=$(\roman*)$, ref=.$(\roman*)$,wide,  labelindent=0pt]
		\item \label{ass.b.f:1}The function $b:[0,T]\times \RR^\xdim \times  \cP_1(\RR^\xdim)\longrightarrow \RR^m$ is Borel--measurable, continuously differentiable, Lipschitz--continuous uniformly in $t$ and with linear growth in its last two arguments. 
		That is,
		\begin{align*}
			\|b(t, x,\mu) - b(t, x', \mu')\| \le \ell_b\big( \|x-x'\| + \cW_1(\mu,\mu')\big),\; \|b(t, x, \mu)\| &\le \ell_b\Big(1+ \|x\| + \Big(\int_{\RR^\xdim}\|x\|\mu(\d x)\Big) \Big),
		\end{align*}
		for all $t \in [0,T]$, $(x,\mu), (x', \mu') \in \RR^\xdim \times \cP_1(\RR^\xdim)$ and some constant $\ell_b>0$.
		\item \label{ass.b.f:2} The function $f:[0,T]\times \RR^\xdim \times A \times \cP_2(\RR^\xdim)\longrightarrow \RR$ is Borel--measurable and with quadratic growth. That is,
	\begin{equation}\label{eq.qgrowth}
		|f(t, x, a,\mu)|	\le \ell_f\Big(1 +  \|x\|^2 + \|a\|^2 + \int_{\RR^\xdim}\|x\|^2\mu(\d x)\Big),
	\end{equation}
	for all $(t,x,a,\mu) \in [0,T] \times \RR^\xdim \times A\times \cP_2(\RR^\xdim)$ and some constant $\ell_f>0$.
	The function $f$ admits the decomposition 
	\begin{equation*}
	 	f(t, x, a, \mu) = f_1(t, a) + f_2(t, x, \mu),
	 \end{equation*}
	where, $f_2$ is bounded from below and continuously differentiable in the last two arguments, and, $f_1$ is continuously differentiable in the second argument, convex, positive, and {satisfies $f_1(t,0)=0$}. Moreover, there is $p>1$ and $C_1\in\R,C_2>0$, such that 
		$$f_1(t,a)\geq C_1+C_2\|a\|^p.$$
\end{enumerate}
\end{assumption}

With this, we can introduce the Hamiltonian functionals $H_1:[0,T]\times\R^\xdim\longrightarrow \R$, $H_2:[0,T]\times \RR^\xdim\times \RR^\xdim\times \mathcal{P}_2(\RR^\xdim)\longrightarrow \RR^m$ and the associated first-order sensitivity functional $F:[0,T]\times \RR^\xdim\times \RR^\xdim\times \mathcal{P}_2(\RR^\xdim\times \RR^\xdim)\longrightarrow \RR^m$, given by
	\begin{equation}
	 	\label{eq:def.H}
	 		H_1 (t, y): = \inf_{a\in A} \big\{ f_1 (t, a ) + a\cdot y\big\}, 
	\end{equation}
$H_2(t,x,y,\mu):=f_2(t,x,\mu)+ b (t, x,\mu )\cdot y$, and,
\begin{equation}
	 	\label{eq:def.F}
	 		F(t,x,y,\nu) :=  \partial_x H_2 (t, x, y,\pi^1_{\#} \nu) 
		+\int_{\RR^\xdim\!\times \R^m}\partial_\mu H_2 (t, \tilde x  ,\tilde y,\pi^1_{\#} \nu )(x) \nu (\d \tilde x\times \d \tilde y) .
\end{equation}

Our analysis of the finite particle approximation of Problem \eqref{eq:problem} in \Cref{sec:finiteSPapprox} requires the following additional set of assumptions.

\begin{assumption}\label{ass.F}	
\begin{enumerate}[label=$(\roman*)$, ref=.$(\roman*)$,wide,  labelindent=0pt]
	\item \label{ass.F:0} The map $a\longmapsto f_1(t,a)$ is uniformly strongly convex with constant $\lambda_{f_1}>0$.

	\item \label{ass.F:1} The map $ (x,\mu)\longmapsto f_2(t,x,\mu)$ is uniformly Lipschitz smooth, and $F$ is Lipschitz--continuous uniformly in $t$ and with linear growth in its last three arguments with Lipschitz--constant $\ell_F>0$.
\end{enumerate}
\end{assumption}

\begin{remark}\label{rmk.assumptions}
	Let us comment on the previous set of assumptions. 
	\begin{enumerate}[label=$(\roman*)$, ref=.$(\roman*)$,wide,  labelindent=0pt]
	\item \label{rmk.assumptions.1} {\rm \Cref{ass.b.f}\ref{ass.b.f:1}} imposes Lipschitz and growth assumptions of the interaction term $b$ so that \eqref{eq:control.SDE.MF} admits a unique strong solution $X^\alpha$. 
	The growth conditions of $f$ in {\rm \Cref{ass.b.f}\ref{ass.b.f:2}} are motivated by our baseline examples, see {\rm \Cref{remark:examples}}.
	{In particular, $a\longmapsto \partial f_1(t,a)$ has linear growth, see {\rm \cite[Theorem 6.7]{evans2015measure}, and the infimum in \eqref{eq:def.H} is attained on a compact set}. 
	Our choice of the $\Wc_1$ norm for the measure term in $b$, as well as the convexity and coercivity conditions of $f_1$, would allow us to obtain the existence of optimizers and {\rm MFSBs} without having to impose compactness on the set {\color{black}$A$} where the controls take values. 
	In general, to use $\Wc_{p^\prime}$ instead on $\Wc_1$, it suffices to take $p>p^\prime \geq1$.

\item \label{rmk.assumptions.3} {\rm \Cref{ass.F}} strengthens the regularity of the cost function, which is easily verified in our benchmark example introduced in {\rm \Cref{remark:examples}}, see {\rm \Cref{sec:proofexistenceentropic}}.
	In particular, we remark for future reference that under {\rm \Cref{ass.F}\ref{ass.F:0}}\footnote{See {\rm \citeauthor{rockafellar2009variational} \cite[Proposition 12.60]{rockafellar2009variational}} and {\rm \citeauthor{rockafellar1970convex} \cite[Theorem 23.5]{rockafellar1970convex}}.}  
	\begin{align}\label{eq.lambda}
	y\longmapsto H_1(t,y) \text{ is differentiable and } \Lambda:[0,T] \times \mathbb{R}^\xdim  \longrightarrow A,\, \Lambda(t, y) := \partial_y H_1(t, y),
	\end{align}  
is Borel--measurable and Lipschitz--continuous and with linear growth in the last argument with Lipschitz--constant $\lambda_{f_1}^{-1}>0$.
	Finally, as we will make use of the maximum principle for McKean-Vlasov control problems, {\rm \Cref{ass.F}\ref{ass.F:1}} imposes standard linear growth and Lipschitz conditions.}
	
\item \label{rmk.assumptions.2} {The above setting encompasses cost function beyond the quadratic cost. For instance, possible choices of $f_1$ are $f_a(x):=C_2 |x|^p$, $p\in(1,2)$, $f_b(x):=x^2+\log (x+\sqrt{1+x^2})$, $f_c (x):=\int_{x^o}^x h(y)\d y$ for $h$ nondecreasing, of at most linear growth and continuous with points of non-differentiability, and $f_d(x):=Mx^2+h(x)$ for any $h$ twice differentiable with second derivative bounded by $M$. Note $f_b$ and $f_d$ are strongly convex.
}
\end{enumerate}
\end{remark}

Let us now explain our strategy to analyse Problem \eqref{eq:problem}. In the language of stochastic control, the terminal cost of the stochastic optimal transport problem \eqref{eq:problem} is the function $\chi_{\mufin}(\cL(X^\alpha_T))$, i.e., the convex indicator of the set $\{\mufin\}$ which is defined as $0$ when $\cL(X^\alpha_T) = \mufin$ and infinity otherwise, see \cite[Section 6.7.3]{carmona2018probabilisticI}. 
	This terminal cost is clearly not differentiable, and consequently the problem is not amenable to standard techniques of stochastic control theory.
	Thus, the starting point of our analysis is a penalization procedure allowing to write the optimal transport problem \eqref{eq:problem} as a ``standard'' McKean--Vlasov control problem.\medskip

\begin{definition}\label{def:penalization}
We say that $g:\Pc_2(\R^m)\longrightarrow [0,\infty)$ is a smooth penalty function if 
\begin{enumerate}[label=$(\roman*)$, ref=.$(\roman*)$,wide,  labelindent=0pt]
	\item $g(\mu)=0$ if and only if $\mu=\mufin;$ \label{def:penalization:i}
	\item $g$ is L-differentiable. \label{def:penalization:ii}
\end{enumerate}
\end{definition}
	
	Given a smooth penalty function $g$, for each $k\ge 1$, we consider the penalized problem
\begin{equation}
\label{eq:inf.problem-k.penalized}
	V^{k}(\muin, \mufin) := \inf \bigg\{ \EE\bigg[\int_0^Tf(t, X^\alpha_t, \alpha_t,\mathcal{L}(X^\alpha_t))\diff t + kg(\mathcal{L}(X^\alpha_T)) \bigg]: \alpha \in \mathfrak{ A},\, X_0^{\alpha}\sim \muin\bigg\},
\end{equation}
where $X^\alpha$ is the unique strong solution of the controlled stochastic differential equation \eqref{eq:control.SDE.MF}.\medskip

Let us now discuss Problem \eqref{eq:inf.problem-k.penalized} and \Cref{def:penalization}; and provide some examples of penalty functions. 
To begin with, note that \eqref{eq:inf.problem-k.penalized} is a standard stochastic control problem of McKean--Vlasov type (also called mean--field control problem). 
Regarding the notion of smooth penalty function, our choice of \Cref{def:penalization}\ref{def:penalization:i} is motivated by the above discussion on the terminal cost of the stochastic optimal transport problem \eqref{eq:problem}, but also by the following duality argument. Letting 
\[
	J(\alpha,k):=\EE\bigg[\displaystyle \int_0^Tf(t, X^\alpha_t, \alpha_t,\mathcal{L}(X^\alpha_t))\diff t + kg (\mathcal{L}(X^\alpha_T))\bigg],
\]
it is clear from {\rm \Cref{def:penalization}\ref{def:penalization:i}} that 
\[ 
	V(\muin, \mufin)=\displaystyle \inf_{\alpha\in \mathfrak{A}} \sup_{k \geq 1} J(\alpha,k), \text{ because }\,  \sup_{k \geq 1} J(\alpha,k ) =
	\begin{cases} 
		\displaystyle \EE\bigg[\int_0^Tf(t, X^\alpha_t, \alpha_t,\mathcal{L}(X^\alpha_t))\diff t \bigg] & \text{if } \mathcal{L}(X^\alpha_T)=\mufin \\
	\infty & \text{otherwise.}
	\end{cases}
\]
Moreover, by weak duality 
\begin{align}\label{eq:weakduality}
	V(\muin, \mufin)\geq \displaystyle \sup_{k\geq 1} V^{k}(\muin,\mufin),
\end{align}
and it is natural that as $k$ goes to infinity, the penalized problems converge to the stochastic optimal transportation problem, i.e., equality in \eqref{eq:weakduality} holds $($see {\rm \Cref{eq:lemma.k}} for details$)$. 		This will be crucial in our construction of solutions.
On the other hand, the smooth property of a penalty function, \Cref{def:penalization}\ref{def:penalization:ii}, will allow us to employ the stochastic maximum principle to analyze Problem \eqref{eq:inf.problem-k.penalized}.
 
\medskip

Regarding the choice of penalty function, at first glance, and by looking at the distance-like condition in \Cref{def:penalization}\ref{def:penalization:i}, it seems appropriate to consider the mapping $g_1:\cP_1(\RR^\xdim)\ni \mu \longmapsto\mathcal{W}_1(\mu , \mufin)$. 
Unfortunately, 
 this is not is a smooth function on the Wasserstein space, thus violating \Cref{def:penalization}\ref{def:penalization:ii}.
	Alternatively, we could have considered the (perhaps more natural) penalty function $g_2:\cP_2(\RR^\xdim)\ni \mu \longmapsto\mathcal{W}_2^2(\mu , \mufin)$.
	{However, as noted by {\rm \citeauthor*{alfonsin2020squared} \cite{alfonsin2020squared}}, this function is differentiable only when $\mufin$ is a Dirac mass, which would conflict with the assumption $V(\muin,\mufin)<\infty$.}
	One possible avenue to circumvent this would be to make use of the mollification theory for functions on the Wasserstein space $\cP_1(\RR^\xdim)$ recently developed by {\rm \citeauthor*{mou2022wellposedness} {\rm \cite[Theorem 3.1]{mou2022wellposedness}}}, see also \citeauthor*{cosso2023smooth} \cite{cosso2023smooth}.
	Indeed, since the map $g_1$ is continuous and 1-Lipschitz on $\cP_1(\RR^\xdim)$, there exists a sequence $g^n_1:\cP_1(\mathbb{R}^\xdim) \to \RR$ of smooth and 1-Lipschitz functions that converges to $g_1$ locally uniformly.
	This would have nevertheless led to an extra approximation level in our approach. 
	A candidate that bypasses this approximation is given by the Fourier-Wasserstein metric on $(\Pc_1(\R^\xdim),\Wc_1)$, 
	recently introduced in \citeauthor*{soner2022viscosity} \cite{soner2022viscosity} in the study of viscosity solutions to the dynamic programming equation. 
	At the intuitive level, the Fourier-Wasserstein metric $\hat \rho$ corresponds to a smooth equivalent of $\Wc_1$ and leads to the penalty function $g_3:\cP_1(\mathbb{R}^\xdim) \ni\mu \to \hat \rho(\mu,\mufin)$.
	Another typical example of a penalty function 
	are Gauge-type functions on $(\Pc_2(\R^\xdim),\Wc_2)$, introduced notably by \citeauthor*{cosso2021master} \cite{cosso2021master}, (see also \citeauthor*{bayraktar2022smooth} \cite{bayraktar2022smooth}) in the study of viscosity solutions of the master equation. 
	Similar to the previous example, one may think of a Gauge function $\rho$ as a smooth equivalent of $\Wc_2$ suggesting the choice $g_4:\cP_2(\mathbb{R}^\xdim) \ni\mu \to \rho(\mu,\mufin)$.

\section{Mean field Schr\"odinger problem: Existence and characterization}\label{sec:mfspexistence}

Given the diverse nature of the examples presented in the previous section and the growing literature on PDEs over the space of probability measures, in this section we work with a generic penalty function $g$ in the sense of \Cref{def:penalization}.
	On $(\Omega, \Fc,\F, \P)$, we introduce the spaces
	\begin{itemize}[wide,  labelindent=0pt]
	\item {$\H^2$ of $\R^{\xdim\times\bmdim}$-valued $\F$-predictable processes $Z$ such that $\displaystyle \|Z\|_{\H^2}^2:= \E^\P\bigg[\int_0^T \|Z_t\|^2\diff t\bigg]<\infty$};
	\item {$\S^2$ of $\R^\xdim$-valued $\F$-progressively measurable continuous processes $Y$ such that $\displaystyle\|Y\|_{\S^2}^2:= \E^\P \bigg[\sup_{t\in [0,T]} \|Y_t\|^2 \bigg]<\infty$}.
	\end{itemize}

We now state the main result of this section which covers \Cref{thm:existence.mfsb} as a particular case.

\begin{theorem}
\label{thm:existence.mfsb.body.paper}
	Let {\rm \Cref{ass.b.f}} be satisfied and let $g$ be a penalty function in the sense of {\rm \Cref{def:penalization}}.
	Then, {if the problem feasible, i.e., $V(\muin, \mufin)<\infty$,} the mean field Schr\"odinger problem \eqref{eq:problem} admits a {\rm MFSB} $\hat X$ with associated optimal control $\hat \alpha$. Moreover, 
	\begin{enumerate}[label=$(\roman*)$, ref=.$(\roman*)$,wide,  labelindent=0pt]
	\item $\hat X$ is, up to a subsequence, the limit in law of $(\hat X^k)_{k\ge1}$, 
	where for each $k\ge1$, $(\hat X^k, \hat Y^k, \hat Z^k)\in \S^2\times\S^2 \times\H^2$ satisfies the McKean--Vlasov equation
		\begin{equation}
	\label{eq:fbsde.reg.MFSP.Theorem}
		\begin{cases}
			\diff \hat X^{k}_t = B^k\big(t, \hat X^{k}_t, \hat Y^k_t, \cL(\hat X^{k}_t)\big)\diff t + \sigma \diff W_t, \,  { t\in [0,T]} \\
			\diff \hat Y^{k}_t = -F\big(t,  \hat X^{k}_t, \hat Y^{k}_t, \cL(\hat X^{k}_t,\hat Y^{k}_t)\big) \diff t + \hat Z^{k}_t\diff W_t, \,  { t\in [0,T]} \\
			\hat X^{k}_0\sim \muin,\; \hat Y^{k}_T = k\partial_\mu	g(\cL(\hat X_T^{k}))(\hat X^{k}_T),
		\end{cases}
	\end{equation}
	whereby $F$ is given in {\rm \Cref{eq:def.F}}, $\Lambda^{\! k}(t,y)$ is a measurable minimizer in \eqref{eq:def.H}, and $B^k$ is the function given by 
	\begin{equation}
	\label{eq:def.B}
		B^k(t, x, y, \mu) := \Lambda^{\! k}(t,  y) + b(t,x, \mu ).
	\end{equation}
	\item \label{thm:existence.mfsb.body.paper:ii}
	{If $ (x,\mu)\longmapsto f_2(t,x,\mu)$ is uniformly Lipschitz smooth,}
	then the {\rm MFSB} $\hat X$ and its associated optimal control $\hat \alpha$ satisfy
	\begin{align}\label{eq.charactirizationlimit}
	\hat X \stackrel{\text{d}}{=} \mybar X, \text{ and, } {\color{black} \hat\alpha_t \stackrel{\text{d}}{=} \Lambda(t,\mybar Y_t)},\text{ for almost all } t\in [0,T],
	\end{align}
	for $\Lambda(t,y)$ a minimizer in \eqref{eq:def.H}, $( \mybar X, \mybar  Y,  \mybar Z)$, defined on a probability space $(\mybar \Omega, \mybar \Fc, \mybar \P)$ supporting a Brownian motion $\mybar W$, satisfying
	\begin{equation}
	\label{eq:fbsde.MFSP.charact.corollary}
		\begin{cases}
			\diff \mybar X_t =  B (t,  \mybar X_t,\mybar Y_t, \cL(\mybar X_t) ) \diff t + \sigma \diff \mybar W_t,\, { t\in [0,T]}\\
			\diff   \mybar Y_t = -F(t,\mybar X_t,\mybar Y_t,\Lc(\mybar X_t,\mybar Y_t))  \diff t +  \mybar Z_t\diff \mybar W_t,\,  { t\in [0,T)}\\
			  \mybar X_0\sim \muin,\; \mybar X_T \sim \mufin,
		\end{cases}
	\end{equation}
	and $B$ as in \eqref{eq:def.B} with $\Lambda$. In particular,
\[
\mybar M_t:= \mybar Y_t +\int_0^t F(r,\mybar X_r,\mybar Y_r,\Lc(\mybar X_r,\mybar Y_r))  \diff r,
\]
is a {$(\mybar \P,\mybar \F^{\bar \P})$-martingale on $[0,T)$, where $\mybar\F^{\bar \P}$ denotes the $\mybar\P$-augmented filtration generated by $\mybar W$}.
\end{enumerate}
\end{theorem}
The above result complements the existence statement of a MFSB by: $(i)$ an asymptotic characterization in the case of general cost functionals $f$ satisfying {\rm \Cref{ass.b.f}}, $(ii)$ a characterization of a MFSB in terms of a so-called planning FBSDEs.
{On the one hand, let us mention that the additional assumptions on the cost functional in the second part of the statement of \Cref{thm:existence.mfsb.body.paper} are satisfied by the cost functionals referred to in \Cref{rmk.assumptions} as well as in the entropic case, see \Cref{sec:proofexistenceentropic}.
On the other hand, due to the expected non-uniqueness of MSFBs, we highlight that the above characterization holds for the MFSB $\hat X$ with associated control $\hat \alpha$ obtained in the first assertion of \Cref{thm:existence.mfsb.body.paper}.
	In particular, it is clear from \eqref{eq:fbsde.reg.MFSP.Theorem} that the approximating FBSDE sequence depends on the choice of penalty function $g$, which is given and fixed in the statement.
	That is, for any penalty function $g$, we can construct a sequence satisfying \eqref{eq:fbsde.reg.MFSP.Theorem} that, up to a subsequence, will converge to the MFSB $\hat X$ with associated control $\hat \alpha$.
	It is then natural to ask whether any minimizer of Problem \eqref{eq:problem} can be obtained as a limit in this form for some penalty function $g$.
	We leave this question as the subject of future research.}
	Nevertheless, the asymptotic characterization is in itself interesting; for instance, for numerical simulation.
In fact, while estimating an MFSB seems to be a very difficult task, the above result suggests that one could obtain an MFSB by numerically solving the FBSDE \eqref{eq:fbsde.reg.MFSP.Theorem} for increasingly large values of $k$.\footnote{ {This will come with its own challenges since a crucial feature in the analysis is that, in general, the convergence happens along a subsequence.}}
Further notice that the present characterization contrasts with \Cref{eq:fbsde.MFSP.charact.corollary} in that \Cref{eq:fbsde.MFSP.charact.corollary} does not seem amenable to numerical approximation because it does not have boundary conditions, whereas \Cref{eq:fbsde.reg.MFSP.Theorem} is a standard McKean--Vlasov FBSDE for which numerical schemes exist.
	In particular, under suitable assumptions, see {\rm \citeauthor*{chassagneux2014probabilistic} \cite{chassagneux2014probabilistic}}, this FBSDE is well-posed. 
This observation will be explored thoroughly in a separate work to present a numerical scheme for the computation of Schr\"odinger bridges.

\subsection{Preparatory lemmas}\label{sec:lemmasexistence}

As mentioned above, our analysis takes full advantage of the fact that Problem \eqref{eq:inf.problem-k.penalized} is a standard stochastic control problem of McKean--Vlasov type. 
	The literature on McKean--Vlasov control problems has quickly developed in recent years, notably due to its link to mean-field games.
	We refer the reader for instance to \citeauthor*{carmona2013probabilistic} \cite[Chapter 6]{carmona2013probabilistic}; \citeauthor*{carmona2015forward} \cite{carmona2015forward}; \citeauthor*{djete2019mckean} \cite{djete2019mckean}, \citeauthor*{cosso2021master} \cite{cosso2021master}; \citeauthor*{soner2022viscosity} \cite{soner2022viscosity} and \citeauthor*{pham2017dynamic} \cite{pham2017dynamic}.

The proof of {\rm \Cref{thm:existence.mfsb.body.paper}} will need the following two intermediate lemmas.

\begin{lemma}
\label{lemma:infty.problem-k.penalized}
	Let {\rm \Cref{ass.b.f}} hold.
	For every $k \ge 1$ the control problem \eqref{eq:inf.problem-k.penalized} with value $V^{k}(\muin,\mufin)$ admits an optimizer $\hat\alpha^{k}$ satisfying $\hat\alpha^{k}_ t= \Lambda^{\! k} (t, \hat Y^{k}_t )$, $\d t\otimes \d \P\text{\rm --a.e.}$, for $\Lambda^{\! k}(t,y)$ a measurable minimizer in \eqref{eq:def.H}, and $(\hat X^{k}, \hat Y^{k}, \hat Z^{k})\in \S^2\times\S^2 \times\H^2$ solving the McKean--Vlasov equation
	\begin{equation}
	\label{eq:fbsde.reg.MFSP}
		\begin{cases}
			\diff\hat X^{k}_t = B(t, \hat X^{k}_t,\hat Y^{k}_t, \cL(\hat X^{k}_t))\diff t + \sigma \diff W_t,\, t\in [0,T],\\
			\diff \hat Y^{k}_t = -F(t, \hat X^{k}_t, \hat Y^{k}_t, \cL(X^{k}_t,Y^{k}_t)) \diff t + \hat Z^{k}_t\diff W_t,\, t\in [0,T],\\
			\hat X^{k}_0\sim \muin,\quad \hat Y^{k}_T = k\partial_\mu	g(\cL(\hat X_T^{k}))(\hat X^{k}_T).
		\end{cases}
	\end{equation}	
\end{lemma}
\begin{proof}
	In a first step, we extract a candidate optimizer. For every positive integer $n$ and $k\geq 1$, there is $\alpha^{k,n} \in \mathfrak{A}$ such that
	\begin{equation}
	\label{eq:lower.bound.Vinfkn}
		V^{k}(\muin, \mufin) \ge \EE\bigg[\int_0^Tf(t, X^{\alpha^{k,n}}_t, \alpha^{k,n}_t,\cL(X^{\alpha^{k,n}}_t))\diff t + kg(\cL(X^{\alpha^{k,n}}_T)) \bigg] - \frac1n.
	\end{equation}
	Since $V^{k}(\muin,\mufin)\le V(\muin, \mufin)$ and the functions $f_2$ and $g$ are bounded from below, there is $C>0$ independent of $n$ such that
	\begin{equation}
	\label{eq:bound.f1}
		\EE\bigg[\int_0^Tf_1(t, \alpha^{k,n}_t)\diff t\bigg] \le C.
	\end{equation}
	Thus, since $f_1$ is convex and coercive, and \eqref{eq:bound.f1} holds, by \Cref{lem:BaLaTa}, 
		{ the sequence $(\int_0^\cdot \alpha^{k,n}_t\diff t)_{n\geq 1}$ is tight and, there exists $\hat \alpha^k$ such that, up to a subsequence, $(\int_0^\cdot \alpha^{k,n}_t\diff t)_{n\geq 1}$ converges to $\int_0^\cdot \hat \alpha_t^k\diff t$ in law in $C([0,T],\RR^m)$} and by convexity of $f_1$ it holds
	\begin{equation*}
	  	\liminf_{n\to \infty}\EE\bigg[\int_0^Tf_1(t, \alpha^{k,n}_t)\diff s\bigg]\ge \EE\bigg[\int_0^Tf_1(t, \hat \alpha_t^k)\diff t\bigg].
	\end{equation*} 
	Note that this last inequality shows that $\hat\alpha^k\in \mathfrak{A}$.\medskip
	
	With this, we now proceed to show the weak convergence of the associated controlled processes. We claim that $(X^n)_{n\geq 1}$, where $X^n:=X^{\alpha^{k,n}}$, is tight. For this note that there is $p>1$ and $C>0$ such that
	\begin{align*}
	\|X_t^n\|^p& \leq  C\bigg(\|X_0\|^p+ \int_0^T \| \alpha_r^{k,n}\|^p\diff r + \int_0^t \| b(r,X_r^n,\Lc(X_r^n))\|^p \diff r+ \|\sigma W_t\|^p\bigg)\\
	& \leq   C\bigg(1+ \|X_0\|^p+ \int_0^T f_1(r, \alpha_r^{k,n}) \diff r + \int_0^t \big(  \|X_r^n\|^p +\E[\|X_r^n\|^p] \big) \diff r+ \|\sigma W_t\|^p\bigg),
	\end{align*}
	{\color{black} where we use $f_1(t,a)\geq C_1+C_2\|a\|^p$}. It then follows from \eqref{eq:bound.f1} and Gronwall's inequality that there is $C>0$ such that
	\begin{align}\label{eq:boundXntight}
	\sup_{n\geq 1} \E\bigg[\sup_{t\in [0,T]} \|X_t^n\|^p\bigg]\leq C.
	\end{align}
	
	Similarly, for the same $p>1$, we may find that there is $C>0$ such that
	\begin{align*}
	\|X_t^n-X_s^n\|^p \leq  C\bigg( |t-s|^{p-1}  \int_0^T \Big( f_1(r, \alpha_r^{k,n})  +\|X_r^n\|^p+\E[\|X_r^n\|^p]\Big) \diff r+ |t-s|^{\frac p2}\|\sigma\|^p\bigg),
	\end{align*}
	which, together with \eqref{eq:bound.f1}, \eqref{eq:boundXntight} and Aldous criterion, see \cite[Theorem 16.11]{kallenberg2002foundations}, allows us to conclude that $(X^n)_{n\geq 1}$ is tight.\medskip

	We are now in a position to establish the weak convergence of $(X^{n})_{n\geq 1}$. 
	By Skorokhod's extension theorem, we can find a common probability space $(\mybar \Omega, \mybar \cF, \mybar \P)$ supporting $(\mybar X^n,\int_0^\cdot \mybar \alpha_t^{k,n} \diff t, \mybar W^n) \stackrel{d}{=} (X^n,\int_0^\cdot \alpha_t^{k,n} \diff t, W)$, for all $n\geq 1$, and $(\mybar X,\int_0^\cdot \mybar \alpha_t^{k} \diff t,\mybar W)\stackrel{d}{=} (X,\int_0^\cdot \hat \alpha_t^{k} \diff t, W)$, and on which $( \mybar X^{n}, \int_0^\cdot \mybar\alpha^{k,n}_t\diff t , \mybar W^n) \longrightarrow  (\mybar X, \int_0^\cdot \mybar{\alpha}_t^k\diff t,\mybar W)$, $\mybar \P\text{--a.s.}$ 
	Let us note that by dominated convergence, $(\int_0^\cdot \mybar \alpha^{k,n}_t\diff t)_{n\geq 1}$ converges to $\int_0^\cdot {\mybar{\alpha}}_t^k\diff t$ in $\mathbb{L}^1(\mybar  \Omega,\mybar  \cF, \mybar \P)$.\medskip
	
	We now note that since $b$ is Lipschitz--continuous and $(\mybar X^n,\int_0^\cdot \mybar \alpha_t^{k,n} \diff t, \mybar W^n) \stackrel{d}{=} (X^n,\int_0^\cdot \alpha_t^{k,n} \diff t, W)$, $\mybar X^n$ solves, under $\mybar \P$, the controlled SDE
	 \begin{equation*}
	 	\diff \mybar X^{n}_t = \mybar \alpha_t^{k,n} + b(t, \mybar X^{n}_t, \cL(\mybar X^{n}_t))\diff t + \sigma\diff \mybar W_t^n,\quad \mybar X^{n}_0 \sim \muin.
	 \end{equation*}
	
	Letting $\mybar \E$ denote the expectation under $\mybar \P$, the Lipschitz continuity of $b$ and Gronwall's inequality leads to
	 \begin{equation*}\label{eq:estimateX.alpha}
	 	\mybar \EE\bigg[\sup_{t\in [0,T]} \|\mybar X^{n}_t - \mybar X^{\mybar{\alpha}^k}_t \|\bigg] \le \mathrm{e}^{2 \ell_bT}\mybar \EE \bigg[\|\mybar X_0^n-\mybar X_0 \|+\sup_{t\in [0,T]} \bigg\|\int_0^t\mybar \alpha^{k,n}_s\diff s - \int_0^t\mybar {\alpha}_s^k\diff s\bigg\|+  \|\sigma\| \sup_{t\in [0,T]}\|\mybar W^n-\mybar W \|  \bigg].
	 \end{equation*}
	 Consequently, $\Wc_1\big(\Lc\big(\mybar X^{n}\big),\Lc\big(\mybar X^{\bar \alpha^k}\big)\big)=\Wc_1\big(\Lc\big( X^{n}\big),\Lc\big( X^{ \alpha^{k}}\big)\big)$ converges to zero as $n$ goes to infinity.
	  {Thus, $(X^{n})_{n\ge0}$ converges in law to $\hat X^{k}:=X^{\alpha^k}$, where $\hat X^{k}$ is uniquely given by the controlled SDE}
	 \begin{equation*}
	 	{\diff \hat X^{k}_t = \hat \alpha_t^k + b(t, \hat X^{k}_t, \cL(\hat X^{k}_t))\diff t + \sigma\diff W_t,\quad \hat X^{k}_0 \sim \muin.}
	 \end{equation*}

	 Therefore, taking the limit in $n$ in \eqref{eq:lower.bound.Vinfkn}, it follows by continuity of $f$ and $g$ that
	 \begin{equation*}
	 	V^{k}(\muin, \mufin) \ge \EE\bigg[\int_0^Tf(t, \hat X^{k}_t, \hat \alpha_t^k,\cL(\hat X^{k}_t))\diff t + kg(\cL(\hat X^{k}_T)) \bigg],
	 \end{equation*}
	showing that $\hat \alpha^k$ is optimal since it is admissible.
	\medskip

	At last, since $a\longmapsto f_1(t, a) +  a\cdot y$ is convex and $A$ is closed and convex, we note that the characterization follows from the maximum principle, see \cite[Theorem 6.14]{carmona2018probabilisticI}.
	{Recall that given $\hat X^k$, the processes $(\hat Y^{k}, \hat Z^k)\in\S^2 \times\H^2$ are uniquely defined by the solution to the BSDE in \eqref{eq:fbsde.reg.MFSP}. Thus, the optimal control $\hat\alpha^k$ satisfies}
	\begin{equation}\label{eq.optimalityalpha}
		{ f_1(t, \hat\alpha_t^{k}) +\hat\alpha^{k}_t  \cdot   \hat Y^{k}_t   = H_1(t, \hat Y^{k}_t) , \;  \d t\otimes \d \PP \text{--a.e.}}
	\end{equation}
	{By a classical measurable selection argument, see \cite[Theorem 3]{schal1974selection},
	there is $\Lambda^{\! k}:[0,T]\times \R^\xdim\longrightarrow A$ measurable such that $\hat\alpha^{k}_t = \Lambda^{\! k}(t,\hat Y^{k}_t)$, $\d t\otimes \d \PP \text{--a.e.}$}\qedhere
\end{proof}

\begin{lemma}
\label{eq:lemma.k}
	Under {\rm \Cref{ass.b.f}} we have
	\begin{equation*}
		\lim_{k \to \infty}V^{k}(\muin, \mufin)= V(\muin, \mufin).
	\end{equation*}
Moreover, strong duality holds, i.e., there is equality in \eqref{eq:weakduality}.
\end{lemma}

\begin{proof}
	Note that the sequence $(V^{k}(\muin, \mufin))_{k\ge0}$ is monotone increasing. Thus, the second statement follows from the first. 
	Let us show the first result. By weak duality, see \eqref{eq:weakduality}, since $V^{k}(\muin,\mufin)\le V(\muin, \mufin)<\infty$, we have that
	\begin{equation*}
		\lim_{k \to \infty}V^{k}(\muin, \mufin) \le V(\muin, \mufin).
	\end{equation*}
	Let us show the reverse inequality.
	{Recalling that $g$ is deterministic, for every positive $k$, there is $\alpha^k\in \mathfrak{A}$ such that}
	\begin{equation}
	\label{eq:upper.bound.lim.k}
		{V^{k}(\muin, \mufin) \ge \EE\bigg[\int_0^Tf(t, X^{\alpha^k}_t, \alpha^k_t,\cL(X^{\alpha^k}_t))\diff t  \bigg] + kg(\cL(X^{\alpha^k}_T))- \frac1k.}
	\end{equation}
	Since $V^{k}(\muin,\mufin)\le V(\muin, \mufin)<\infty$ and the functions $g$ and $f_2$ are bounded from below, as in the proof of \Cref{lemma:infty.problem-k.penalized}, the sequence $(\int_0^\cdot \alpha^{k}_t\diff t)_{k\geq 1}$ is tight and converges to $\int_0^\cdot \hat \alpha_t\diff t$ in law, with $\hat \alpha\in \mathfrak A$, and consequently, $X^{\alpha^k}$ is tight and converges to $X^{\hat \alpha}$ in law. Moreover, there is $C>0$ independent of $k$ such that
	 \begin{align}
	 	g(\cL(X^{\alpha^k}_T)) &\le \frac1k\bigg(V^{ k}(\muin,\mufin)- \EE\bigg[\int_0^Tf(t, X^{\alpha^k}_t, \alpha^k_t,\cL(X^{\alpha^k}_t))\diff t  \bigg] \bigg) + \frac{1}{k^2}\notag \\
	 	&\le \frac1k\Big(V(\muin,\mufin)- C \Big) + \frac{1}{k^2}\xrightarrow [k\to \infty]{}0. \label{eq.lemmak.continuity}
	\end{align}
	Thus, by definition of $g$, $ X_T^{\hat \alpha}\sim\mufin$. Now, since $g$ is non-negative we have, thanks to the continuity of $f$, back in \eqref{eq:upper.bound.lim.k} that 
	\begin{align*}
		\lim_{k\to \infty}  V^{k}(\muin, \mufin) & \ge \lim_{k\to \infty} \EE\bigg[\int_0^Tf(t, X^{\alpha^k}_t, \alpha^k_t,\cL(X^{\alpha^k}_t))\diff t  \bigg] - \frac1k \geq  \EE\bigg[\int_0^Tf(t, X_t^{\hat \alpha}, \hat \alpha_t,\cL(X^{\hat \alpha}_t))\diff t  \bigg] .
	\end{align*}
	Since $\hat\alpha\in \mathfrak A$ and $X_T^{\hat \alpha}\sim \mufin$, we find that
	\[
	\lim_{k\to \infty} V^{k}(\muin,\mufin)\geq 
	V(\muin,\mufin),
	\] 
	which yields the result.
\end{proof}

\subsection{The general case: Proof of Theorem {\rm \ref{thm:existence.mfsb.body.paper}}}

\begin{proof}[Proof of {\rm \Cref{thm:existence.mfsb.body.paper}}]
Let us argue the first statement and $(i)$ together. 
By \Cref{lemma:infty.problem-k.penalized} for every $k\geq 0$ there exists $\hat\alpha^k\in \mathfrak{A}$ with associated optimal state process $\hat X^{k}$ and adjoint process $\hat Y^{k}$ satisfying \eqref{eq:fbsde.reg.MFSP.Theorem} which is an optimizer of $V^{k}$, i.e.,
	\begin{equation}
	\label{eq:optimizer.Vk}
		V^{k}(\muin, \mufin) = \EE\bigg[\int_0^Tf(t,\hat X^{k}_t, \hat\alpha^{k}_t, \cL(\hat X^{k}_t))\diff t + kg(\cL(\hat X^k_T)) \bigg].
	\end{equation}
	Once again, since $f_2, g$ are bounded from below and $V^{k}(\muin, \mufin)\le V(\muin,\mufin) <\infty$ we find that
	\begin{equation}\label{eq:bound.f1.alphak}
		\EE\bigg[\int_0^Tf_1(t, \hat\alpha^k_t)\diff t\bigg] \le C,
	\end{equation}
	for some constant $C>0$.
	Thus, from \Cref{lem:BaLaTa} {we obtain that there exists $\hat \alpha$ such that, up to a subsequence, $\hat A^k(\cdot):=\int_0^\cdot\hat \alpha^{k}_t\diff t $ converges to $\hat A(\cdot):=\int_0^\cdot \hat \alpha_t\diff t$ in law in $C([0,T],\R^\xdim)$}, which in turn, arguing as in the proof of \Cref{lemma:infty.problem-k.penalized}, allows us to derive that $(\hat X^{k})_{k\ge1}$ converges in law to {$\hat X:=\hat X^{\hat \alpha}$}. Moreover, as in the proof of \Cref{eq:lemma.k}, since $V^{k}(\muin,\mufin)\le V(\muin, \mufin)<\infty$, {$g$ is nonnegative by definition, and, thanks to \Cref{ass.b.f}, $f$ is bounded from below}, there is $C>0$ independent of $k$ such that
	 \begin{align*}
	 	g(\cL(\hat X^{k}_T)) \le \frac1k\Big(V(\muin,\mufin)- C \Big),
	\end{align*}
	from which we deduce that $\Lc(\hat X_T)=\mufin$. All in all, $(\hat X^{k})_{k\ge1}$ converges in law to $\hat X$ satisfying
		\begin{equation}\label{eq.existencethm.limitX}
			\diff \hat X_t = \hat\alpha_t+ b(t,  \hat X_t, \cL(\hat X_t))\diff t + \sigma \diff W_t, \quad \hat X_0 \sim \muin, \quad \hat X_T \sim \mufin.
		\end{equation}
	
	This establishes $(i)$ provided that we verify the optimality of $\hat \alpha$. Back in \eqref{eq:optimizer.Vk}, we see that
	\begin{align}
	\label{eq:ine.lim.Vk}
	 	\lim_{k\to \infty}V^{k}(\muin,\mufin)\ge \EE\bigg[\int_0^Tf(t, \hat X_t, {\hat \alpha}_t,\cL(\hat X_t))\diff t \bigg]. 
	\end{align}
	Since the reverse inequality in \eqref{eq:ine.lim.Vk} is clear by definition of $V^{k}(\muin, \mufin)$ and the fact that $\hat \alpha\in \mathfrak{A}$, it follows that
	\begin{equation*}
		\lim_{k\to \infty}V^{k}(\muin,\mufin) =
	 	\EE\bigg[\int_0^Tf(t, \hat X_t, {\hat \alpha}_t,\cL(\hat X_t))\diff t \bigg]
	\end{equation*}
	and thus, by  {\rm \Cref{eq:lemma.k}} and uniqueness of the limit, we have
	\begin{equation*}
		V(\muin,\mufin) = \EE\bigg[\int_0^Tf(t, \hat X_t, {\hat \alpha}_t,\cL(\hat X_t))\diff t \bigg].
	\end{equation*}
	That is, $\hat\alpha$ is optimal. \medskip
	
	We now argue $(ii)$. 
	We argue in four steps. 
	\medskip
	
	{\bf Step 1.} To begin with, recall that by $(i)$, as in \Cref{lemma:infty.problem-k.penalized}, the sequence $(\hat X^k,\int_0^\cdot \hat \alpha_t^{k} \diff t,  W)_{k\geq 1}$ is tight. 
	Thus, by Skorokhod extension theorem, there is a common probability space $(\mybar \Omega, \mybar \cF, \mybar \P)$ supporting $(\mybar X^k,\bar A^k , \mybar W^k) \stackrel{d}{=} (\hat X^k,  \hat A^k, W)$, for all $k\geq 1$, and $(\mybar X,\bar A  ,\mybar W)\stackrel{d}{=} (\hat X,  \hat A, W)$, with $\bar A^k(\cdot):=\int_0^\cdot \mybar\alpha^{k}_t\diff t$, $\bar A(\cdot):=\int_0^\cdot \mybar{\alpha}_t\diff t$, and on which $(\mybar X^{k}, \bar A^k , \mybar W^k) \longrightarrow (\mybar X,  \bar A, \mybar W)$, $\mybar \P\text{--a.s.}$~as $k$ goes to infinity. 
	Moreover, $\mybar X$ satisfies \eqref{eq.existencethm.limitX}, $\mybar \P\text{--a.s.}$ We will write $\mybar \E$ for the expectation under $\mybar \P$ and let $\mybar \F^k:=\F^{\overline W^k }$ and $\mybar \F:=\F^{\overline W}$, e.g., $\mybar \F$ denotes the filtration generated by $\mybar W$.  \medskip
	
	To obtain the result, i.e., to arrive at \eqref{eq:fbsde.MFSP.charact.corollary}, we construct $(\mybar X^k, \mybar Y^k, \mybar Z^k)_{k\geq 1}$ on $(\mybar \Omega, \mybar \cF, \mybar \P)$, analog to \eqref{eq:fbsde.reg.MFSP.Theorem}, and study its limit. Note that the previous paragraph defines $(\mybar X^k)_{k\geq 1}$ and identifies its limit. We are left to introduce the pair $(\mybar Y^k, \mybar Z^k)$.\medskip
	
	{On the one hand, since \Cref{eq.optimalityalpha} is equivalent to $Y^k_t=-\partial_a f_1(t,\alpha^k_t)$, $\d t\otimes \d\P \text{--a.e.}$, see \Cref{eq:def.H} and \cite[Theorem 23.5]{rockafellar1970convex}},
	we let 
	$$\mybar Y^k_\cdot:=-\partial_a f_1(\cdot,\mybar \alpha_\cdot^k).$$ 

{\color{black}Notice that for any $k\geq 1$ and any bounded continuous $G:\R^\xdim \longrightarrow \R$ and almost every $t\in [0,T]$, $ \E[G(h^{-1}(\hat A^k(t+h)-\hat A^k(t)))]=\mybar \E[G(h^{-1}(\bar A^k(t+h)-\bar A^k(t)))]\longrightarrow\mybar \E[G(\mybar \alpha_t^k)]$, as $h\longrightarrow 0$. 
	Thus, we have that $\mybar \alpha^k_t\stackrel{d}{=} \hat \alpha^k_t$ for almost every $t\in [0,T]$. 
	Thus, $\mybar Y_t^k\stackrel{d}{=} \hat Y^k_t$ for almost every $t\in [0,T]$, $k\geq 1$.} 
	We now claim that, for $p>1$ given by \Cref{ass.b.f}\ref{ass.b.f:2}, there exists $\mybar Y\in\L^p([0,T]\times \mybar \Omega)$ such that $\mybar Y^k$ converges to $\mybar Y$ weakly in $\L^p([0,T]\times \mybar \Omega)$. 
	Indeed, since $f_1$ satisfies {\color{black} $f_1(t,a)\geq C_1+C_2\|a\|^p$}, {and $a\longmapsto \partial f_1(t,a)$ has linear growth}, see \Cref{rmk.assumptions}\ref{rmk.assumptions.1}, we have that
	\[
	\mybar \E \bigg[\int_0^T \|\mybar Y_t^k\|^p \diff t\bigg]=\mybar \E \bigg[\int_0^T \|\partial_a f_1(t,\mybar \alpha_t^k)\|^p \diff t\bigg]\leq \ell_{f_1}\mybar  \E \bigg[\int_0^T(1+\| \mybar \alpha_t^k \|^p )\diff t\bigg]\leq C,
	\]
	 where the last inequality follows from \eqref{eq:bound.f1.alphak}.
	Consequently, $(\mybar Y^k)_{k\geq 1}$ (and $(\mybar \alpha^k)_{k\geq 1}$), being a bounded sequence in a reflexive normed space, contain weakly convergent subsequences, see \citeauthor*{ambrosio2000functions} \cite[Theorem 1.36]{ambrosio2000functions}. 
	That is, up to a subsequence, $(\mybar Y^k)_{k\geq 1}$ converges weakly in $\L^p([0,T]\times \mybar \Omega)$ to $\mybar Y\in\L^p([0,T]\times \mybar \Omega)$.  
	{It follows from \cite[Theorem 24.4]{rockafellar1970convex} and the previous construction that $\mybar Y_t=-\partial_a f_1(t,\mybar \alpha_t)$, $\d t\otimes \d \mybar\P\text{--a.e.}$, and we may use \cite[Theorem 23.5]{rockafellar1970convex} again to conclude $\mybar \alpha_t =\Lambda(t,\mybar Y_t)$}.\footnote{{Indeed, $f_1$ is continuous (relative to $A$) and thus closed in the sense that $f={\rm cl} f$, where ${\rm cl} f$ is the lower semi-continuous hull of $f$, see \cite[Theorem 7.1]{rockafellar1970convex} and the discussion thereafter.}}\medskip

	On the other hand, since $(\hat X^k,\hat Y^k,\hat Z^k)$ satisfies \eqref{eq:fbsde.reg.MFSP.Theorem}, and $\Omega$ is a Polish space, we have thanks to \cite[Lemma 1.13]{kallenberg2002foundations} that there exists measurable functions $\Phi^k:[0,T]\times C([0,T],\R^\xdim)\longrightarrow \R^{\xdim\times n}$, such that $\hat Z_t^k=\Phi^k(t,W_{\cdot \wedge t})$ on $(\Omega,\Fc)$ where $W_{\cdot \wedge t}$ denotes the path of $W$ up to time $t$. 
	Let $\mybar Z_t^k:=\Phi^k(t,\mybar W^k_{\cdot \wedge t})$ and note that by monotone convergence
	\begin{align}\label{eq.charact.intZ}
	\overline \E\bigg[ \int_0^T \|\mybar Z_t^k\|^2\diff t\bigg]= \E\bigg[ \int_0^T \|\hat  Z_t^k\|^2\diff t\bigg]<\infty.
	\end{align}
	By discretization, we have that
	\begin{align}\label{eq.convk.disc}
	\begin{split}
	& \mybar \E\bigg[\bigg\| \mybar Y_t^k-\mybar Y_0^k+\int_0^t F(r,\mybar X_r^k,\mybar Y_r^k,\Lc(\mybar X_r^k,\mybar Y_r^k))\diff r-\int_0^t \mybar Z_r^k \diff \mybar W_r^k\bigg\|^2\bigg] \\
	&=\E \bigg[\bigg\| \hat Y_t^k-\hat Y_0^k+\int_0^t F(r,\hat X_r^k,\hat Y_r^k,\Lc(\hat X_r^k,\hat Y_r^k))\diff r-\int_0^t \hat Z_r^k \diff W_r\bigg\|^2\bigg]=0.
	\end{split}
	\end{align}
	{Consequently, by definition of $\mybar Z^k$ and \eqref{eq.charact.intZ},} we have that the sequence $(\mybar M^k)_{k\geq 1}$ given by 
\[
\mybar M^k_t:= \mybar{Y_t}^{k}+\int_0^t\Big(  \partial_x H_2(r,\mybar X_r^k, \mybar Y_r^k ,\Lc(\mybar X_r^k))   +  \tilde \E \big[ \partial_\mu H_2 (r,\widetilde{ X}_r^k, \widetilde{Y}_r^k ,\Lc(\mybar X_r^k))(\mybar  X_r^k)  \big] \Big)   \diff r,
\]
is a $(\mybar \P,\mybar \F^k)$--martingale for every $k$, where $\widetilde{ X}^k$ and $\widetilde Y^k$ denote i.i.d.\,copies of $\mybar X^k$ and $\mybar Y^k$.
 \medskip

{\bf Step 2.} We claim that there is $D\subseteq[0,T]$ dense in $[0,T]$, such that for every $t\in D$, $(\mybar M^k_t)_{k\geq 1}$ converges weakly in $\L^1(\mybar\Omega)$ to 
\[
\mybar M_t:= \mybar Y_t +\int_0^t \Big( \partial_x H_2(r,\mybar X_r,\mybar Y_r,\Lc(\mybar X_r))   +  \tilde \E \big[\partial_\mu H_2 (r,\widetilde{ X}_r, \widetilde{ Y}_r,\Lc(\mybar X_r))(\mybar X_r) \big] \Big)  \diff r.
\]
{To show the claim, we prove that each of the summands converges weakly to its respective limit. First, notice that $\mybar M$ is bounded in $\L^p([0,T]\times\Omega)$. Thus, the existence of the dense set $D\subseteq[0,T]$ is given by \Cref{lem:weakLpprojection}. Note that this also guarantees the claim for the term $(\mybar Y^k)_{k\geq 1}$ of the sequence.}\medskip

Let $t\in [0,T]$ be fixed. We now consider the second term. Note that there exists $C>0$ and $q>0$, $1/p+1/q=1$, such that
\begin{align*}
&\mybar \E\bigg[ \mybar K \int_0^t   \partial_x H_2(r,\mybar X_r^k, \mybar Y_r^k ,\Lc(\mybar X_r^k))  \diff r- \mybar K \int_0^t  \partial_x H_2(r,\mybar X_r,\mybar Y_r,\Lc(\mybar X_r))  \diff r\bigg]\\
 & =\mybar \E \bigg[ \mybar K \int_0^t  \Big(\partial_x b(r,\mybar X_r^k,\Lc(\mybar X_r^k))- \partial_x b(r,\mybar X_r,\Lc(\mybar X_r))\Big)  \mybar Y_r^k\diff r \bigg]+ \mybar\E \bigg[  \int_0^t \mybar K \partial_x b(r,\mybar X_r,\Lc(\mybar X_r))   ( \mybar Y_r^k- \mybar Y_r ) \diff r\bigg] \\
&\;\; + \mybar\E \bigg[ \mybar K \int_0^t  \big(\partial_x f_2(r,\mybar X_r^k,\Lc(\mybar X_r^k))- \partial_x f_2 (r,\mybar X_r,\Lc(\mybar X_r))\big)   \diff r \bigg] \\
&\leq  C  \mybar\E  \bigg[\int_0^t  \| \partial_x b(r,\mybar X_r^k,\Lc(\mybar X_r^k))- \partial_x b(r,\mybar X_r,\Lc(\mybar X_r))\|^q\diff r\bigg]^{\frac1q} + \mybar\E \bigg[  \int_0^t \mybar K \partial_x b(r,\mybar X_r,\Lc(\mybar X_r))  ( \mybar Y_r^k- \mybar Y_r ) \diff r\bigg] \\
&\;\; + \mybar\E \bigg[ \mybar K \int_0^t  \big(\partial_x f_2(r,\mybar X_r^k,\Lc(\mybar X_r^k))- \partial_x f_2 (r,\mybar X_r,\Lc(\mybar X_r))\big)   \diff r \bigg],
\end{align*}
where the inequality follows from Cauchy-Schwarz, the boundedness of $\mybar K$, and the fact that $(\mybar Y^k)_{k\geq 1}$ is bounded in $\L^p([0,T]\times \mybar \Omega)$.
	Now, recall that under {\rm \Cref{ass.b.f}\ref{ass.b.f:1}} the mapping $(x,\mu)\longmapsto b (t,x,\mu)$ is continuously differentiable and Lipschitz, and under the extra assumptions in $(ii)$ $(x,\mu)\longmapsto f_2 (t,x,\mu)$ is Lipschitz smooth uniformly in $t\in[0,T]$.
	Thus, by dominated convergence, the almost sure convergence of $\mybar X^k$ to $\mybar X$ and of $\Lc(\mybar X_t^k)$ to $\Lc(\mybar X_t)$ for $t\in [0,T]$, implies that the first and third terms tend to zero. 
	The second term tends to zero in light of the boundedness of $\partial_x b$ and $\mybar K$ and the convergence of $\mybar Y^k$ to $\mybar Y$ weakly in $\L^1([0,T]\times \mybar \Omega)$.\medskip

We are left to argue for the last term in $\mybar M^k$. Again, for $\mybar K$ bounded, there exists $C>0$ and $q>0$, $1/p+1/q=1$, such that
\begin{align*}
&\mybar \E  \bigg[ \mybar K \int_0^t \widetilde{ \E} \big[ \partial_\mu    H_2(r,\widetilde{ X}_r^k, \widetilde{ Y}_r^k ,\Lc(\mybar X_r^k))(\mybar X_r^k) \big] \d r- \mybar K\int_0^t    \widetilde \E \big[\partial_\mu   H_2(r,\widetilde{ X}_r,\widetilde{ Y}_r,\Lc(\mybar  X_r))(\mybar X_r^k)\big] \d r\bigg]\\
&=   \mybar \E  \bigg[  \int_0^t \mybar K \widetilde \E \Big[ \big(  \partial_\mu b(r,\widetilde{ X}_r^k,\Lc(\mybar X_r^k))(\mybar X_r^k)- \partial_\mu b(r,\widetilde{ X}_r,\Lc(\mybar X_r))(\mybar X_r)\big)   \widetilde{ Y}_r^k  \Big]  \d r\bigg]  +\mybar\E \bigg[ \mybar K \int_0^t \widetilde \E\big[  \partial_\mu b(r,\widetilde{ X}_r,\Lc(\mybar X_r))(\mybar X_r)   ( \widetilde{ Y}_r^k- \widetilde{ Y}_r )\big] \d r\bigg]\\
&\;\;  +\mybar\E  \bigg[ \mybar K \int_0^t   \widetilde \E\big[ \partial_\mu f_2(r,\widetilde{ X}_r^k,\Lc(\mybar X_r^k))(\mybar X_r^k)- \partial_\mu f_2 (r,\widetilde{ X}_r,\Lc(\mybar X_r))(\mybar X_r) \big]  \d r \bigg] \\
&\leq  C \mybar\E \bigg[  \int_0^t  \widetilde \E \big[ \|   \partial_\mu b(r,\widetilde{ X}_r^k,\Lc(\mybar X_r^k))(\mybar X_r^k)- \partial_\mu b(r,\widetilde{ X}_r,\Lc(\mybar X_r))(\mybar X_r)\|^q \big] \d r\bigg]^{\frac{1}q}   +\widetilde \E\bigg[  \int_0^t  \mybar\E \big[\mybar K  \partial_\mu b(r,\widetilde{ X}_r,\Lc(\mybar X_r))(\mybar X_r)\big]   ( \widetilde{ Y}_r^k- \widetilde{ Y}_r ) \diff r\bigg]\\
&\;\;  + \mybar\E \bigg[ \mybar K \int_0^t   \widetilde \E\big[ \partial_\mu f_2(r,\widetilde{ X}_r^k,\Lc(\mybar X_r^k))(\mybar X_r^k)- \partial_\mu f_2 (r,\widetilde{ X}_r,\Lc(\mybar X_r))(\mybar X_r) \big]  \d r \bigg] ,
\end{align*}

where we used the boundedness of $\mybar K$ and of $(\widetilde Y^k)_{k\geq 1}$ in $\L^p([0,T]\times \widetilde \Omega)$ in the first term, and Fubini's theorem to exchange the order of the expectations in the second term.
	Arguing as above, the convergence of the second term follows from that of $(\widetilde{Y}^k)_{k\geq 1}$, whereas the first and third terms exploit the convergence of $(\mybar X^k)_{k\geq 1}$, $(\widetilde X^k)_{k\geq 1}$ and $(\Lc(\mybar X_t^k))_{k\geq 1}$ for $t\in [0,T]$ with the assumptions on $b$ and $f_2$. With this, the claim is proved.\medskip

{\bf Step 3.} We show that $(\mybar M_t)_{t\in D}$ is a $(\mybar \P,\mybar \F)$-martingale for $D$ given in {\bf Step 2}. Let $0\leq s\leq t\leq T$, $s,t\in D$, and $\mybar K$ be a bounded $\mybar \Fc_s$--measurable random variable and note that since $\mybar M^k$ is a $(\mybar \P,\mybar \F^k)$--martingale we have that
\begin{align*}
\mybar \E\big[ \mybar K\big( \mybar \E[\mybar M_t|\mybar\Fc_s]-\mybar M_s\big)\big]&=\mybar \E\big[\mybar K\big( \mybar \E[\mybar M_t|\mybar\Fc_s]-\mybar \E[\mybar M_t|\mybar\Fc_s^k]\big)\big]+\mybar \E\big[ \mybar K\big( \mybar \E[\mybar M_t|\mybar\Fc_s^k]-\mybar \E[\mybar M_t^k|\mybar\Fc_s^k]\big)\big]+\mybar \E\big[ \mybar K\big( \mybar \E[\mybar M_t^k|\mybar\Fc_s^k]-\mybar M_s\big)\big]\\
&= \mybar \E\big[ \mybar K\big( \mybar \E[\mybar M_t|\mybar\Fc_s]-\mybar \E[\mybar M_t|\mybar\Fc_s^k]\big)\big]+\mybar \E\big[ \mybar K\big( \mybar \E[\mybar M_t- \mybar M_t^k|\mybar\Fc_s^k]\big)\big]+\mybar \E\big[ \mybar K\big( \mybar M_s^k -\mybar M_s\big)\big].
\end{align*}
Now, note that the convergence of $(\mybar M^k_t)_{k\geq 1}$ weakly in $\L^1(\mybar \Omega)$ for every $t\in D$, implies the last term goes to zero. 
	Let us now look at the first term. 
	Since $\mybar W^k$ converges to $\mybar W$, $\mybar \P\text{--a.s.}$, $\mybar W^k$ converges to $\mybar W$ in probability in the $J_1(\R^\bmdim)$ topology. 
	Thus, by definition of $\mybar \F^k$ we have thanks to \citeauthor*{coquet2001weak} \cite[Proposition 2]{coquet2001weak} that $\mybar \F^k$ converges weakly to $\mybar \F$ in the sense of \cite[Definition 2]{coquet2001weak}, and consequently, $\mybar \E[\mybar M_t|\mybar\Fc_\cdot^k]$ converges to $\mybar \E[\mybar M_t|\mybar\Fc_\cdot]$ in probability. 
	Since $\mybar\E[\|\mybar M_t\|^p]<C$ for $t\in D$, by de la Vallée-Poussin's theorem, we have that the family $\mybar \E[\mybar M_t|\mybar\Fc_\cdot^k]$ is uniformly integrable, which in light of the convergence in probability implies that $\mybar \E[\mybar M_t|\mybar\Fc_\cdot^k]$ converges to $\mybar \E[\mybar M_t|\mybar\Fc_\cdot]$ in $\L^1(\mybar \Omega)$ and thus the first term goes to zero for any $t\in D$. 
	We are only left to study the convergence of the second term. For this, we notice that since $\mybar \F^k$ converges weakly to $\mybar \F$, we have that $\mybar \Fc^k_s$ converges weakly to $\mybar \Fc^k_s$ in the sense of \cite[Definition 1]{coquet2001weak}, i.e., $\mybar K^k:=\E[\mybar K|\mybar \Fc^k_s]$ converges to $\mybar K$ in probability. Since $\mybar K$ is bounded, by de la Vallée-Poussin's theorem, $\mybar K^k$ converges to $\mybar K$ in $\L^1(\mybar \Omega)$, but, by the boundedness of $\mybar K$, also in $\L^q(\mybar \Omega)$, $1/q+1/p=1$. Thus,
\begin{align*}
\mybar \E\big[ \mybar K\big( \mybar \E[\mybar M_t- \mybar M_t^k|\mybar\Fc_s^k]\big)\big]&=\mybar \E\big[ (\mybar K-\mybar K^k)\big( \mybar \E[\mybar M_t- \mybar M_t^k|\mybar\Fc_s^k]\big)\big]+\mybar \E\big[ \mybar K^k \big( \mybar M_t- \mybar M_t^k\big)\big]\\
&\leq \|\mybar K-\mybar K^k\|_{\L^q(\overline \Omega)}\| \mybar M_t- \mybar M_t^k\|_{\L^p(\overline \Omega)}+\mybar \E\big[ \mybar K^k \big( \mybar M_t- \mybar M_t^k\big)\big],
\end{align*}
where we used the fact that $\mybar K^k$ is $\mybar \Fc^k_s$-measurable. Thus, the first term goes to zero since $\mybar\E[\|\mybar M_t\|^p],\mybar\E[\|\mybar M_t^k\|^p]<C$, for all $k\geq 1$ and $t\in D$ and the convergence of the sequence $(\mybar K^k)_{k\geq 1}$. The second term goes to zero in light of the convergence of $(\mybar M^k_t)_{k\geq 1}$ weakly in $\L^1(\mybar \Omega)$.\medskip

{\bf Step 4.} We now define for $t\in [0,T)$ the process\footnote{To be precise, $D\cap\Q$ stands for any countable dense subset of $D\cap [0,T]$.}
\[
\mybar M^+_t:=\lim_{s\in D\cap\Q ,\,   s\downarrow t}\,  \mybar M_s,
\]
which thanks to \citeauthor*{dellacherie1982probabilities} \cite[Theorem VI.2]{dellacherie1982probabilities} is well defined and a supermartingale with respect to $\mybar \F^+$ the right limit of $\mybar \F$.
	We now claim that $t\longmapsto \E[\mybar M_t^+]$ is constant.
	Note that the claim follows from
	\[
	\mybar \E[\mybar M_t^+]=\mybar \E\Big[ \lim_{s\in D\cap\Q,  s\downarrow t} \mybar M_s\Big]=\lim_{s\in D\cap\Q,  s\downarrow t} \mybar \E[\mybar M_s]=\mybar \E[\mybar M_0],
	\]
	if we show that for any $(s_n)_{n\geq0}\subseteq D\cap\Q$, $t\leq s_{n+1}\leq s_n$, the family $(\mybar M_{s_n})_{n\geq 0}$ is uniformly integrable.
	But this follows from \cite[Theorem V.30]{dellacherie1982probabilities}, since letting $\mybar N:=(\mybar N_n)_{n\leq 0}$ and $\G:=(\Gc_{n})_{n\leq 0}$ be given by $\mybar N_{-n}:=\mybar M_{s_n}$ and $\mybar \Gc_{-n}:=\mybar \Fc_{s_n}$ for $n\geq 0$, we obtain that $\mybar N$ is a discrete backward $(\mybar\P,\mybar \G)$-martingale and $\mybar \E[\mybar N_n]=\mybar \E[\mybar M_{0}]$, $n\leq 0$. 
	Thus, since adding the $\mybar \P$-null sets does not affect the supermartingale property, by \cite[Theorem V.30]{dellacherie1982probabilities} we conclude that $(\mybar M_t^+)_{t\in [0,T)}$ is a $(\mybar \P,\mybar \F^{\bar \P})$-martingale, where $\mybar \F^{\bar \P}$ denotes the $\mybar \P$-augmented filtration generated by $\mybar W$. This establishes the last statement in $(ii)$.
	Since $\mybar \F^{\bar \P}$ satisfies the usual hypothesis, by martingale representation, see \citeauthor*{jacod2003limit} \cite[Theorem III.4.29]{jacod2003limit}, we find that there exists a predictable and locally square integrable process $\mybar Z$, such that
\[
 \mybar Y_t^+ =\mybar Y_s^+ +\int_t^s F(r,\mybar X_r,\mybar Y_r^+,\Lc(\mybar X_r,\mybar Y_r^+))   \diff r -\int _t^s \mybar Z_r\diff \mybar W_r, \, 0\leq t \leq s<T,
\]
where $\mybar Y^+$ denotes the right-continuous modification of $\mybar Y$, and $F$ is given by \eqref{eq:def.F}. In particular, $\mybar \alpha_t =\Lambda(t,\mybar Y_t^+), \d t\otimes\d \P\text{--a.e.}$ establishing \eqref{eq.charactirizationlimit}. This implies that $(\mybar X, \mybar Y^+, \mybar Z)$ solves the system \eqref{eq:fbsde.MFSP.charact.corollary}.
\end{proof}

\subsection{The at most quadratic cost case: Proof of Theorem {\rm \ref{thm:existence.mfsb}}}\label{sec:proofexistenceentropic}

At last, we bring ourselves back to the case of the transportation cost problem presented in the introduction and its specialization to the entropic case.\medskip

In the setting of the mean field entropic transportation cost problem we have that $f(t,x,a,\mu)=f_1(t,a)$ and $f_2\equiv 0$, which readily verifies {\rm \Cref{ass.b.f}}{\rm \ref{ass.b.f:2}} and the smooth Lipschitz assumption in \Cref{thm:existence.mfsb.body.paper}\ref{thm:existence.mfsb.body.paper:ii}, and $b(t,x,a,\mu)=a-\nabla \Psi \ast \mu (x)$ so that 
the Lipschitz assumption on $\nabla\Psi$ guarantees that {\rm \Cref{ass.b.f}}{\rm \ref{ass.b.f:1}} holds. 
Notice also that $\Lambda(t, y)=-y$ and
\[
F(t,x,y,\nu) =-\int_{\RR^\xdim\times \R^m} \big(  \nabla^2 \Psi ( x -\tilde x)\cdot y -\nabla^2 \Psi (\tilde x-x) \cdot  \tilde  y \big) \nu (\d \tilde x\times \d \tilde y).
\]
Thus, $(i)$ and $(ii)$ follow directly from \Cref{thm:existence.mfsb.body.paper}\ref{thm:existence.mfsb.body.paper:ii}.
	Lastly, note that in $(iii)$ we have that $H_1(t,y)=-\|y\|^2/2$,
  $\Lambda(t, y)=-y$ so the result is a direct consequence of the symmetry of $\Psi$.

\section{Finite particle approximation of the mean field Schr\"odinger problem}\label{sec:finiteSPapprox}
We now turn to the approximation of the mean field Schr\"odinger problem discussed in the previous section by the Schr\"odinger problem over a finite number of interacting. 
	Let us begin by introducing the generic finite particle problem we will consider.\medskip

Fix $N \in \mathbb{N}$.
Let $\cA$ be the set of admissible controls defined as 
\begin{equation*}
	\cA := \bigg\{\alpha:[0,T]\times \Omega \to A,\,\, \FF^N\text{--progressive processes such that } \EE\bigg[\int_0^Tf_1(t ,\alpha_t)  \diff t\bigg]<\infty\bigg\}.
\end{equation*}
We use the notation
\begin{equation}\label{eq.N.vector}
	\balpha^N:= (\alpha^{1,N},\dots,\alpha^{N,N}),\quad \text{and}\quad \X^{N,\balpha}:= (X^{1,\alpha^N},\dots, X^{N,\alpha^N}),
\end{equation}
for any control processes $\alpha^{i,N} \in \cA$, $i\in \{1,\dots,N\}$, and controlled diffusion
\begin{equation}
\label{eq:controlled.SDE.NP}
	X^{i,\balpha^N}_t = X_0^{i} + \int_0^t\alpha_u^{i,N} + b(u, X^{i,\balpha^N}_u,L^N(\X^{N,\balpha}_u))\diff u + \sigma W^i_t,
\end{equation}
{where $X_0^{1},\dots, X_0^{N}$ are i.i.d.~$\Fc_0$-measurable $\R^\xdim$-valued random variables.} 
	Recall the notation $L^N(\x)$ for the empirical measure of the vector $\x$ as introduced at the end of \Cref{sec:intro}.
We consider the stochastic optimal transport problem
\begin{align}\label{eq:N.problem}
	V^N(\muin, \mufin) := \inf\bigg\{\E\bigg[\frac1N\sum_{i=1}^N\int_0^Tf(t, X^{i,\balpha^N}_t, \alpha_t^{i,N},L^N(\X^{N,\balpha}_t))\diff t\bigg]: \balpha^N  \in \cA^N, \,   X_0^{i}\sim \muin , \, X_T^{i,\alpha^N} \sim \mufin  , \,1\leq i\leq N  \bigg\}.
\end{align}
This problem can be seen as a natural generalization of the entropic stochastic optimal transport problem.
Here, we are optimally steering \emph{interacting} particles with initial configuration $\muin$ to the terminal configuration $\mufin$.

\begin{remark}
	Let us remark that appropriate choices of $f$ and $b$ make $V^N(\muin, \mufin)$ coincide with the value $\Vc^N_{\rm e}(\muin,\mufin)$ of {\eqref{eq:N.SP.intro}} and $\Vc_{\rm S}(\muin,\mufin)$ of \eqref{eq.dyn.SP.control}, {and that consistent with our analysis in {\rm \Cref{sec:mfspexistence}}, the assumption $V^N(\muin, \mufin)<\infty$, $N\in\N$ is in place}.
\end{remark}

As explained in the introduction, the original motivation of this work was to show convergence of $(V^{N}(\muin,\mufin))_{N\geq 1}$ to $V(\muin,\mufin)$.
	This can be done only under the assumption that there exists a displacement convex penalty function for $\mufin$, see \Cref{rmk:convergence}.
	Since the existence of such functions is not clear and depends on $\mufin$,
	in the following, we consider \emph{weaker} versions of these problems, $V^{N}_c(\muin,\mufin)$ and $V_c(\muin,\mufin)$, for which no additional conditions beyond {\rm Assumptions \ref{ass.b.f}} and {\rm \ref{ass.F}} are necessary to establish the corresponding convergence $(V^{N}_c(\muin,\mufin))_{N\geq 1}$ to $V_c(\muin,\mufin)$, of the sequence of weak values.
	The key feature behind this result is the combination of a penalization procedure together with the additional convexity in the space of probability measures that is inherent to the weak problems.
	We stress that the weak versions of the problem are interesting in their own right as explained below.

\subsection{Finite particle approximation under displacement convexity}\label{sec.approx.disp.conv}

Our result exploits the notion of displacement convexity, which we present next, to show the finite particle approximation.

\begin{definition}\label{def.dispconv}
An $L$-differentiable function $\phi: \Pc(\R^\xdim) \longrightarrow \R$ is said to be displacement convex if there is $\lambda_\phi\geq 0$ such that for any two $\mu_0, \mu_1 \in \Pc_2(\R^\xdim)$ and $\mu\in \Gamma(\mu_0,\mu_1)$ we have that
\[
\int_{\R^\xdim\times\R^\xdim} (x-y)\cdot \big(\partial_\mu\phi(\mu_1)(x)-\partial_\mu\phi(\mu_0)(y)\big)\mu(\d x, \d y)\geq \lambda_\phi \int_{\R^\xdim\times\R^\xdim} \|x-y\|^2 \mu(\d x, \d y).
\]
\end{definition}

It turns out we can naturally embed this property into the analysis if we consider a weaker version of $V(\muin, \mufin)$.
	Given a random variable $X$ with finite expectation and $\nu\in \Pc_1(\R^m)$, we say
	 \[ X \msim \nu, \text{ whenever } \Lc(X)\leq_{\rm cvx}\nu,\] 
	where $\leq_{\rm cvx}$ denotes the convex order of $\Pc_1(\R^\xdim)$.\footnote{For $\mu,\nu\in \Pc_1(\R^\xdim)$, $\mu\leq_{\rm cvx}\nu\Longleftrightarrow  \mu(\varphi)\leq \nu(\varphi) $ for every $\varphi:\R^\xdim\longrightarrow \R$ convex.}
	With this, for $X^\alpha$, $\mathfrak{A}$ and $f$ as in the formulation of $V(\muin, \mufin)$, we consider 
\begin{equation}
\label{eq:problem.convex}
	V_{c}(\muin, \mufin) := \inf\bigg\{\E\bigg[\int_0^T f(t, X^\alpha_t, \alpha_t,\mathcal{L}(X^\alpha_t))\diff t \bigg]: \alpha \in \mathfrak{A},\,  X_0^\alpha\sim\muin,\, X_T^\alpha \msim   \mufin \bigg\},
\end{equation}
and, in complete analogy to \eqref{eq:N.problem}, introduce the same terminal condition on the processes $X^{i,\alpha^N}$ given by \eqref{eq:controlled.SDE.NP}, i.e.,
\begin{align}\label{eq:N.problem.c}
	V^N_c(\muin, \mufin) := \inf\bigg\{\E\bigg[\frac1N\sum_{i=1}^N\int_0^Tf(t, X^{i,\balpha^N}_t, \alpha_t^{i,N},L^N(\X^{N,\balpha}_t))\diff t\bigg]: \balpha^N  \in \cA^N, \,   X_0^{i}\sim \muin , \, X_T^{i,\alpha^N} \msim \mufin  , \,1\leq i\leq N  \bigg\}.
\end{align}

Thus, $V_c(\muin, \mufin)$ seeks to find the ``best trajectory'', in the sense of minimal cost, to transport the probability distribution $\muin$ to a terminal configuration $\Lc(X_T^\alpha)$ that is dominated by $\mufin$ in the convex order.

\begin{remark}
	By {\rm \citeauthor{strassen1965existence} \cite{strassen1965existence}}, $\mu\leq_{\rm cvx}\nu$ if and only if $\exists \pi\in\Gamma_M(\mu,\nu)$, a so-called martingale coupling, where\footnotemark
	\[ \Gamma_M(\mu,\nu):=\bigg\{ \pi=\mu\times\pi_x\in \Pi(\mu,\nu): \int_{\R^m} y \pi_x(\d y)= x, \text{ for } \mu\text{--a.e. } x\in \R^m \bigg\}.\]
\end{remark}
\footnotetext{Here, $\pi=\mu\times\pi_x$ is understood as $\pi(\d x,\d y)=\mu(\d x)\pi_x(\d y)$ where $(\pi_x)_{x\in \R^m}$ denotes the regular conditional desintegration of $\pi$ with respect to its first marginal. Note that the elements of $\Gamma_M(\mu,\nu)$ correspond to the law of a martingale $(M_t)_{t\in\{ 0,1\}}$.}

The following is the main result of this section and consists of a convergence result that gives conditions guaranteeing that the value of the finite particle weak problem converges to the value of its corresponding mean field weak Schr\"odinger problem, i.e., the convergence of the sequence $(V^{N}_c(\muin,\mufin))_{N\geq 1}$ to $V_c(\muin,\mufin)$.

\begin{theorem}
\label{thm:convergence.c.mfsb.body.paper}
	Let {\rm Assumptions \ref{ass.b.f}} and {\rm \ref{ass.F}} hold.
	{Suppose $V_c(\muin, \mufin) <\infty$ and $V^N_c(\muin, \mufin)<\infty$}.
	Then we have that
	\begin{equation*}
		\lim_{N\to \infty}V^N_c(\muin,\mufin) = V_c(\muin,\mufin).
	\end{equation*}
\end{theorem}

\begin{remark}\label{rmk:convergence}
{We remark that our analysis leads to the following convergence results for the values of {\rm Problem \ref{eq:N.problem}} to its mean-field counterpart in {\rm Problem \ref{eq:problem}}, so long as they are feasible, i.e., $V^N(\muin, \mufin)<\infty$ and $V(\muin, \mufin)<\infty$. 
	If $\mufin$ admits a displacement convex penalty function $g$ in the sense of {\rm Definitions \ref{def:penalization}} and {\rm \ref{def.dispconv}}, then
	\begin{equation*}
		\lim_{N\to \infty}V^N(\muin,\mufin) = V (\muin,\mufin).
	\end{equation*}
Indeed, as we will elaborate in {\rm \Cref{rmk.convergenceN.mainproblem}}, the displacement convexity assumption handles the crux of the argument in the proof of {\rm \Cref{thm:convergence.c.mfsb.body.paper}}. 
	Nevertheless, since verifying whether $\mufin$ admits a displacement convex penalty function is, potentially, a demanding task, we believe a different proof technique will be essential.
	We leave this as the subject of further research and refer to {\rm \Cref{rmk.convergenceN.regularity}} for a discussion on the unviability of two alternative approaches.
}
\end{remark} 

In the remainder of this section the assumptions of {\rm \Cref{thm:convergence.c.mfsb.body.paper}} are in place unless otherwise stated. In particular, see {\rm \Cref{rmk.assumptions}\ref{rmk.assumptions.3}}, the map $\Lambda(t,y)$ is unique and given by \eqref{eq.lambda}, and second, by {\rm \cite[Proposition 12.60]{rockafellar2009variational}}, we have that there is $\lambda_{f_1}>0$ such that, for all $t\in [0,T]$, $y,y^\prime\in \R^\xdim$, 
	\begin{align}\label{eq:displaassumpLambda}
(y - y^\prime )\cdot(\Lambda(t,y) - \Lambda(t,y^\prime)) \le - \lambda_{f_1}^{-1} \|y-y^\prime\|^2.
\end{align}

In light of the previous discussion, we set
\begin{equation}
	\label{eq:def.B.conv}
		B(t, x, y, \mu) := \Lambda(t,  y) + b(t,x, \mu ).
\end{equation}

Here again, the argument starts with the introduction of an appropriate penalization of the stochastic optimal transport problems $V^N_c(\muin, \mufin)$ and $V_c(\muin, \mufin)$ in terms of the dual representation of the convex ordering. This will lead us to consider one extra level in the penalization but will allow us to exploit the inherent convexity of the weak problems.\medskip

By \citeauthor{gozlan2017kantorovich} \cite[Proposition 3.2]{gozlan2017kantorovich} and \citeauthor{azagra2013global} \cite{azagra2013global}, letting $\Phi:=\{ \varphi:\R^\xdim\longrightarrow \R$ convex, continuously differentiable, 1-Lipschitz, bounded from below$\}$, we have that 
\begin{align}\label{eq.dualitywot.bary}
	X \msim \nu \Longleftrightarrow  \Wc_{c}(\Lc(X),\nu )=0, \text{ where, }  \Wc_{c}(\mu,\nu ): =\sup_{\varphi \in\Phi}\ \big\{  \mu( \varphi ) -   \nu(\varphi)\big\}.
\end{align}

Thus, we introduce the functions $g_c$ and $g^\varphi_c$, for $\varphi\in \Phi$, by 
\begin{align*} 
g_c:\Pc_1(\R^\xdim)\longrightarrow [0,\infty), \, \mu\longmapsto \Wc_c (\mu,\mufin), \text{ and, }g^\varphi_c:\Pc_1(\R^\xdim)\longrightarrow\R, \, \mu\longmapsto \mu(\varphi)-\mufin(\varphi).
\end{align*}

The next lemma justifies how the weaker versions embed convexity into the problem.

\begin{lemma}\label{lemma.reg.gphi}
Let $\varphi\in \Phi $. Then, $\mu\longmapsto g^\varphi_c(\mu)$ is $L$-differentiable and displacement convex with derivative $\partial_\mu g^\varphi_c(\mu)(x)= \nabla \varphi(x)$.
\begin{proof} Recall that for a real-valued continuously differentiable function $f$ with a derivative of at most linear growth, the mapping $\phi:\mu\longmapsto\mu(f)$ is $L$-differentiable and $\partial_\mu \phi (\mu)(x)= \nabla f(x)$, see \cite[Section 5.2]{carmona2018probabilisticI}.  
The $L$-convexity follows from the convexity and continuous differentiability of $\varphi\in \Phi$.\qedhere
\end{proof}
\end{lemma}

With this for $\alpha\in \mathfrak{A}$, $k\geq 1$, and $\varphi\in \Phi$, we set
\[
J_c(\alpha,k):=\E\bigg[\int_0^T f(t, X^\alpha_t, \alpha_t,\mathcal{L}(X^\alpha_t))\diff t \bigg]+k g_c(\Lc(X_T^\alpha) ),\;\text{and, }  J_c(\alpha,k,\varphi):=\E\bigg[\int_0^T f(t, X^\alpha_t, \alpha_t,\mathcal{L}(X^\alpha_t))\diff t \bigg] +k  g^\varphi_c(\Lc(X_T^\alpha) ),
\]
which give rise to the problems $V_c^k(\muin, \mufin)$ and $V_c^{k,\varphi}(\muin, \mufin)$ given by
\begin{equation}
\label{eq:problem.kpenalized}
	V_c^k(\muin, \mufin) := \inf\big\{J_c(\alpha,k): \alpha \in \mathfrak{A},\,  X_0^\alpha\sim\muin  \big\},\;\text{and, } V_c^{k,\varphi}(\muin, \mufin) := \inf\big\{J_c(\alpha,k,\varphi) : \alpha \in \mathfrak{A},\,  X_0^\alpha\sim\muin  \big\}.
\end{equation}

Similarly, we consider the penalized control problems
\begin{equation}
\label{eq:N.problem.c-k.penalized}
	V^{N,k}_c(\muin, \mufin) := \inf \bigg\{ \frac1N\sum_{i=1}^N\EE\bigg[\int_0^Tf(t, X^{i,\alpha^N}_t\!\!, \alpha_t^{i,N},L^N(\X^{N, \balpha}_t))\diff t + kg_c(\mathcal{L}(X^{i,\alpha^N}_T)) \bigg]: \alpha^N \in \mathcal{A}^N,\,   X_0^{i}\sim \muin ,\, 1\leq i\leq N\bigg\},
\end{equation}
and,
\begin{equation}
\label{eq:N.problem.c-kphi.penalized}
	V^{N,k,\varphi}_c(\muin, \mufin) := \inf \bigg\{ \frac1N\sum_{i=1}^N\E\bigg[\int_0^T \! \! f(t, X^{i,\alpha^N}_t\!\!, \alpha_t^{i,N},L^N(\X^{N, \balpha}_t))\diff t + kg^\varphi_c(\mathcal{L}(X^{i,\alpha^N}_T)) \bigg]: \alpha^N \in \Ac^N,\,   X_0^{i }\sim \muin ,\, 1\leq i\leq N\bigg\},
\end{equation}
where $X^{i,\alpha^N}$ is the unique strong solution of the controlled stochastic differential equation \eqref{eq:controlled.SDE.NP}.\medskip

As in \Cref{eq:lemma.k}, we obtain the following duality and characterizations, see \Cref{appendix}.

\begin{lemma}\label{lemma.conv.existence}
\begin{enumerate}[label=$(\roman*)$, ref=.$(\roman*)$,wide,  labelindent=0pt]
\item \label{lemma.conv.existence.i} There exists $\hat \alpha\in \mathfrak{A}$ optimal for $V_c(\muin, \mufin)$. Moreover,
\begin{align}\label{eq.limitink}
V_c(\muin,\mufin)=\lim_{k\longrightarrow \infty} V^{k}_c(\muin,\mufin).
\end{align}

\item \label{lemma.conv.existence.ii} For any $k\geq1,\varphi\in\Phi$, $V_c^{k,\varphi}(\muin, \mufin)$ admits an optimizer $\alpha^{k,\varphi}$, and
\begin{align}\label{eq.supinf.fixedk}
V_c^k(\muin, \mufin) =\sup_{\varphi\in  \Phi} V_c^{k,\varphi}(\muin, \mufin), \;\text{for any } k\geq 1 .
\end{align}
Moreover,
\begin{align}\label{eq.supsupinf}
V_c(\muin, \mufin)= \sup_{k\geq 1} \sup_{\varphi\in  \Phi}  \inf_{\alpha \in \mathfrak{A}} J(\alpha,k,\varphi).
\end{align}
\end{enumerate}
\end{lemma}
\begin{remark}\label{rmk.duality.Nprob}
We remark that a version of {\rm \Cref{lemma.conv.existence}} can be obtained for $V^N_c(\muin, \mufin)$, $V^{N,k}_c(\muin, \mufin) $ and $V^{N,k,\varphi}_c(\muin, \mufin)$.
\end{remark}

The main ingredient for proving the convergence of the sequence $(V^{N, k,\varphi}_c(\muin,\mufin))_{N\geq 1}$ to $V^{k,\varphi}_c(\muin,\mufin)$ consists of a propagation of chaos result for coupled FBSDEs. 
	Indeed, in complete analogy with the analysis in \Cref{sec:lemmasexistence}, the next lemma, whose proof is deferred to \Cref{appendix:characterization.Nparticle}, introduces the FBSDE system associated with Problems $V_c^{k,\varphi}(\muin, \mufin)$ and $V_c^{N,k,\varphi}(\muin, \mufin)$. 

\begin{lemma}\label{lemma.charact.c.k.phi}
Let $\Lambda$, $B$ and $F$ be defined as in {\rm Equations \eqref{eq.lambda}, \eqref{eq:def.B.conv}} and {\rm \eqref{eq:def.F}}, respectively. Let $N,k\ge1$ and $\varphi\in \Phi$ be fixed.
\begin{enumerate}[label=$(\roman*)$, ref=.$(\roman*)$,wide,  labelindent=0pt]
\item \label{lemma.charact.c.k.phi.MF} The control problem with value $V_c^{k,\varphi}(\muin, \mufin)$ admits an optimizer $\hat \alpha^{k,\varphi}$ satisfying $ \hat \alpha^{k,\varphi}_ t= \Lambda (t,  Y^{k,\varphi}_t )$, $\d t\otimes \d \P\text{\rm --a.e.}$, and $(  X^{k,\varphi},  Y^{k,\varphi},  Z^{k,\varphi})\in \S^2\times \S^2\times \H^2$ solving the McKean--Vlasov equation
	\begin{equation}
	\label{eq:fbsde.MFSP.proof.c}
		\begin{cases}
			\diff  X^{k,\varphi}_t = B(t,  X^{k,\varphi}_t, Y^{k,\varphi}_t, \Lc( X^{k,\varphi}_t))\diff t + \sigma \diff W_t,\, t\in [0,T],\\
			\diff  Y^{k,\varphi}_t = -F(t,  X^{k,\varphi}_t,  Y^{k,\varphi}_t, \Lc(X^{k,\varphi}_t,Y^{k,\varphi}_t)) \diff t +  Z^{k,\varphi}_t\diff W_t,\, t\in [0,T],\\
			 X^{k,\varphi}_0\sim \muin,\quad  Y^{k,\varphi}_T = k \nabla \varphi( X^{k,\varphi}_T).
		\end{cases}
	\end{equation}
\item \label{lemma.charact.c.k.phi.Npart}
	The control problem with value $V^{N,k,\varphi}_c(\muin, \mufin)$ admits an optimizer $\hat\balpha^{N,k,\varphi}:= (\hat \alpha^{1,N,k,\varphi},\dots, \hat \alpha^{N,N,k,\varphi})$ satisfying \(\hat\alpha^{i,N,k,\varphi}_t = \Lambda\big(t, Y^{i,N,k,\varphi}_t \big), \,  \d t\otimes \d \P \text{\rm--a.e.}\), $i=1,\dots,N$, and $(X^{i,N,k,\varphi}, Y^{i,N,k,\varphi},Z^{i,j,N,k,\varphi} )\in \S^2\times \S^2\times \H^2$ solving the {\rm FBSDE}\footnote{$\X^{N,k,\varphi}$ and $\Y^{N,k,\varphi}$ are defined in analogy to \eqref{eq.N.vector}, e.g., $\X^{N,k,\varphi}:=(X^{1,N,k,\varphi},\dots,X^{N,N,k,\varphi})$.}
	\begin{equation}
	\label{eq:N.fbsde.proof.main.c}
	\begin{cases}
		\diff X_t^{i,N,k,\varphi} = B\big(t, X_t^{i,N,k,\varphi}, Y^{i,N,k,\varphi}_t, L^N(\X^{N,k,\varphi}_t)\big)\diff t + \sigma \diff W^i_t,\, t\in [0,T],\\
		\diff Y^{i,N,k,\varphi}_t = -F\big(t, X^{i,N,k,\varphi}_t, Y_t^{i,N,k,\varphi}, L^N(\X^{N,k,\varphi}_t,\Y^{N,k,\varphi}_t) \big)\diff t + \sum_{j=1}^NZ^{i,j,N,k,\varphi}_t\diff W_t^j,\, t\in [0,T],\\
		X_0^{i,N,k,\varphi} \sim \muin, \quad Y_T^{i,N,k,\varphi} = k \nabla \varphi (X^{i,N,k,\varphi}_T).
	\end{cases}
	\end{equation}
	In particular, there is $C>0$, independent on $N$, such that $\| X^{i,N,k,\varphi}\|_{\S^2}+\| Y^{i,N,k,\varphi}\|_{\S^2}+\| Z^{i,N,k,\varphi}\|_{\H^2}<C$, $i=1,\dots,N$.
\end{enumerate}
\end{lemma}

The goal now is to show that the particle system {\rm \eqref{eq:N.fbsde.proof.main.c}} evolving forward and backward in time converges to the McKean--Vlasov system \eqref{eq:fbsde.MFSP.proof.c} in an appropriate sense.
	This propagation of chaos would follow from results developed in \cite{lauriere2021convergence,peng2022laplace} if not for the fact that, for $k$ fixed, the $N$-particle system is actually a McKean--Vlasov FBSDE itself, since the terminal value of the value processes $Y^{i,N}$ depends on the law of the forward process. 
	We thus need new arguments to obtain convergence in the present case.
	{We emphasize that, beyond its own mathematical interest, one advantage of considering the weaker problems $V^N_c(\muin, \mufin)$ and $V_c(\muin, \mufin)$ is that their penalized versions $V_c^{k,\varphi}(\muin,\mufin)$ and $V_c^{N,k,\varphi}(\muin,\mufin)$ are given in terms of penalty functions which are displacement convex in the sense of \Cref{def.dispconv}, see \Cref{lemma.reg.gphi}.
	We will exploit this to derive in \Cref{thm.prop.chao.c} a propagation of chaos result which holds uniformly in $\varphi\in \Phi$ and for all $k\geq 1$.
	On the one hand, this will be crucial when we bring ourselves back to the convergence of $V^N_c(\muin,\mufin)_{N\geq 1}$.
	On the other hand, there is no need to make extra assumptions on the measure $\mufin$ beyond the feasibility of the problems.
	Back in the problems $V(\muin,\mufin)$ and $V^N(\muin,\mufin)$, this property is in general absent and as stated in the finite particle approximation result alluded to in {\rm \Cref{rmk:convergence}}, becomes an assumption on the target measure.
	}
	\medskip

As mentioned already, the crux of the finite particle approximation lies in the following propagation of chaos result. For this, we introduce the space $(\Hc^2,\|\cdot\|_{\Hc^2})$ of adapted $\R^\xdim\times\R^\xdim$--valued processes $(X,Y)$ equipped with the norm
\(\| (X,Y)\|_{\Hc^2}^2:=\| X \|_{\S^2}^2+\|Y\|^2_{\H^2}\).\hfill
 \medskip

To simplify the presentation, in the remainder of this section, we drop the superscript $k$ and $\varphi$ from the FBSDEs in the analysis unless otherwise stated.

\begin{theorem}\label{thm.prop.chao.c}
	Let $k\geq 1$ and $\varphi\in \Phi$ be fixed. 
	Let $(  X^{i,N},  Y^{i,N})$ be part of the solution to \eqref{eq:N.fbsde.proof.main.c}
	and let $( \widetilde X^i, \widetilde Y^i)$, $i=1,\dots,N$, denote i.i.d.~copies of $( X, Y, Z)$ solution to \eqref{eq:fbsde.MFSP.proof.c} {driven by $W^i$}.
	Then, there is $C>0$ such that for $i=1,\dots,N$, we have that
\begin{align}\label{eq.propagationresults}
 \| (X^{i,N}-\widetilde X^i,Y^{i,N}-\widetilde Y^i)\|_{\Hc^2}^2  \leq C \E \big[\epsilon^N\big], \text{ where, } \epsilon^N:= \int_0^T  \cW_1^2\big(L^N(\widetilde \X^N_t, \widetilde \Y^N_t), \cL(X_t, Y_t)\big) \diff t .
\end{align}
{\color{black} In particular, $( X^{i,N},Y^{i,N})$ converges to $(  \widetilde X^i,  \widetilde Y^i)$ in $\Hc^2$.}
\end{theorem}

\begin{proof}
Let $\delta X^i_t:= X^{i,N}_t-\widetilde X^i_t$ and $\delta Y^i_t := Y^{i,N}_t - \widetilde Y^{i}_t$.
	Applying It\^o's formula, {since $(\widetilde X^i,\widetilde Y^i)$ solves \eqref{eq:fbsde.MFSP.proof.c} driven by $W^i$}, the dynamics of $\delta X^i_t\cdot \delta Y^i_t$ are, for a martingale $M^i$, given by
\begin{align*}
	\delta X^i_T \cdot \delta Y^i_T- \delta X^i_0 \cdot \delta Y^i_0 &= \int_0^T \delta X^i_t\cdot\Big(-F\big(t, X^{i,N}_t, Y^{i,N}_t, L^N(\X_t^N, \Y^N_t)\big) + F\big(t, \widetilde X^i_t, \widetilde Y^i_t, \cL(X_t, Y_t)\big)\Big)\d t\\
	&\quad  +\int_0^T  \delta Y^i_t\cdot\Big(B\big(t, X^{i,N}_t, Y^{i,N}_t, L^N(\XX_t^N)\big) - B\big(t, \widetilde X^i_t, \widetilde Y^i_t, \cL(X_t)\big)\Big) \diff t + M_T^i, 
\end{align*}

Observe that by \Cref{lemma.reg.gphi}, we have
\begin{align}\label{eq.disp.conv.proof}
	\E[\delta X^i_T \cdot \delta Y^i_T] &= k  \E\Big[\delta X^i_T\cdot \big( \nabla \varphi (X^{N,i}_T) - \nabla \varphi (X_T)\big) \Big] \ge 0.
\end{align}
Hence, using Lipschitz--continuity of $F$ and $b$ (and recalling that $B(t,x,y, \mu) =\Lambda (t,y) + b(t,x,\mu)$) and then Young's inequality, for every $\varepsilon>0$ it holds that
\begin{align*}
	0 &\le \E\bigg[\int_0^T \ell_F \|\delta X_t^i\|\Big( \|\delta X_t^i\|+\|\delta Y_t^i\|+\frac{1}N\sum_{j=1}^N \|\delta X_t^j\|+\frac{1}N\sum_{j=1}^N \|\delta Y_t^j\| + \cW_1\big(L^N(\widetilde \X^N_t, \widetilde \Y^N_t), \cL(X_t, Y_t)\big) \Big)\d t\bigg]\\
	&\quad + \E\bigg[\int_0^T  \ell_b \|\delta Y_t^i\|\Big( \|\delta X_t^i\|+\frac{1}N\sum_{j=1}^N \|\delta X_t^j\| + \cW_1\big(L^N(\widetilde \X^N_t), \cL(X_t)\big)\Big)\d t\bigg]+  \E\bigg[\int_0^T\delta Y^i_t\cdot\big( \Lambda(t, Y^{i,N}_t) - \Lambda (t, \widetilde Y^i_t) \big)\diff t \bigg]\\
	&\le  \E\bigg[\int_0^TC_1^\eps \|\delta X^i_t\|^2 + C_2^\eps \frac1N\sum_{j=1}^N\|\delta X^j_t\|^2\diff t \bigg] + \varepsilon\E\bigg[ \int_0^T\|\delta Y^i_t\|^2 + \frac1N\sum_{j=1}^N\|\delta Y^j_t\|^2\diff t \bigg] - \lambda_{f_1}^{-1} \E\bigg[\int_0^T\|\delta Y^i_t\|^2\diff t \bigg]\\
	&\quad + \E\bigg[\int_0^T\eps \cW_1^2\big(L^N(\widetilde \X^N_t, \widetilde \Y^N_t), \cL(X_t, Y_t)\big) + \frac{\ell_b^2}{\eps} \cW_1^2\big(L^N(\widetilde \X^N_t), \cL(X_t)\big)\diff t \bigg]
\end{align*}
where $C_1^\eps:=\ell_F + (2 \ell_F^2+\ell_b^2)\varepsilon^{-1}$, $C_2^\eps:=\eps \ell_F + \ell_b^2\varepsilon^{-1}$, in the first inequality we use triangular inequality for the Wasserstein distance, and in the second inequality we use \eqref{eq:displaassumpLambda}.\medskip

Thus, rearranging terms, picking $\varepsilon<\lambda_{f_1}^{-1} /2$ and letting $\gamma:=\lambda_{f_1}^{-1} -2\varepsilon>0$ we obtain
\begin{align}
\notag
	  \gamma \E\bigg[\int_0^T\|\delta Y^i_t\|^2\diff t \bigg] &\le \E\bigg[\int_0^TC_1^\eps \|\delta X^i_t\|^2 + C_2^\eps \frac1N\sum_{j=1}^N\|\delta X^j_t\|^2\diff t \bigg] + \varepsilon\E\bigg[ \int_0^T \frac1N\sum_{j=1}^N\|\delta Y^j_t\|^2\diff t \bigg]\\
	&\quad + \E\bigg[\int_0^T\eps \cW_1^2\big(L^N(\widetilde \X^N_t, \widetilde \Y^N_t), \cL(X_t, Y_t)\big) +\frac{\ell_b^2}{\eps}   \cW_1^2\big(L^N(\widetilde \X^N_t), \cL(X_t)\big)\diff t \bigg].\label{eq:estim1.disp.conv1}
\end{align}

Now, applying It\^o's formula to $\|\delta X^i_t\|^2$ and Young's inequality, it follows by the Lipschitz--continuity of $B$ that
\begin{align}\label{eq:estim3.displ.conv}
	\|\delta X^i_t\|^2 &= \int_0^t2\delta X^i_s\cdot\Big( B(s, X^{i,N}_s, Y^{i,N}_s, L^N(\XX_s^N) - B(t, \widetilde X^i_s, \widetilde Y^i_s, \cL(X_s))\Big) \diff s\notag \\
			&\le \ell_B\int_0^t5 \|\delta X^i_s\|^2 + \|\delta Y^i_s\|^2 + \frac1N\sum_{j=1}^N\|\delta X^j_s\|^2 + \cW^2_1\big(L^N(\XX^N_s), \cL(X_s)\big)\diff s.
\end{align}

Applying Gronwall's inequality (resp. averaging and then applying Gronwall's inequality) we find that
\begin{align}
	\E\big[ \|\delta X^i_t\|^2\big] & \le \e^{5\ell_B t}\E\bigg[\int_0^t \|\delta Y^i_s\|^2+ \frac1N\sum_{j=1}^N\|\delta X^j_s\|^2  + \cW^2_1\big(L^N(\widetilde\XX^N_s), \cL(X_s)\big)\diff s\bigg]  \label{eq:estim5.displ.conv1}\\
	\E\Big[\frac1N\sum_{j=1}^N\|\delta X^j_t\|^2\Big] &\le \e^{6 \ell_B t}\E\bigg[\int_0^t\frac1N\sum_{j=1}^N\|\delta Y^j_s\|^2 + \cW^2_1\big(L^N(\widetilde\XX^N_s), \cL(X_s)\big)\diff s\bigg]  .\label{eq:estim4.displ.conv1}
	\end{align}

Let us now average out on both sides of \eqref{eq:estim1.disp.conv1} and use \eqref{eq:estim4.displ.conv1} to obtain
\begin{align*}
	  ( \gamma-\eps) \E\bigg[\int_0^T \frac1N\sum_{i=1}^N \|\delta Y^i_t\|^2\diff t \bigg] &\le   \e^{6\ell_B T} ( C_1^\eps + C_2^\eps)  \E\bigg[ \int_0^T  \int_0^t  \frac1N\sum_{j=1}^N\|\delta Y^j_s\|^2  \diff s \diff t \bigg]    \\
	&\quad + \E\bigg[\int_0^T\eps \cW_1^2\big(L^N(\widetilde \X^N_t, \widetilde \Y^N_t), \cL(X_t, Y_t)\big) +\big( T \e^{6\ell_B T} ( C_1^\eps + C_2^\eps)+ \frac{\ell_b^2}{\eps}\big)   \cW_1^2\big(L^N(\widetilde \X^N_t), \cL(X_t)\big)\diff t \bigg].
\end{align*}

It then follows from Gronwall's inequality, updating $\eps$ so that $\gamma-\eps>0$ and with $\epsilon^N$ as in \eqref{eq.propagationresults}, that there is $C>0$ such that
\begin{align}\label{eq:estim2.displ.conv}
 \E\bigg[\int_0^T \frac1N\sum_{i=1}^N \|\delta Y^i_t\|^2\diff t \bigg] &\le   C \E\big[\epsilon^N\big] ,
 \end{align}
where we used the fact that  $ \cW_1\big(L^N(\widetilde \X^N_t), \cL(X_t)\big)\leq \cW_1\big(L^N(\widetilde \X^N_t, \widetilde \Y^N_t), \cL(X_t, Y_t)\big)$. Consequently, back in \eqref{eq:estim4.displ.conv1} we obtain that there is $C>0$ such that
\begin{align}\label{eq:estim5.displ.conv}
\E\bigg[\frac1N\sum_{j=1}^N\|\delta X^j_t\|^2\bigg]&\le   C \E\big[\epsilon^N\big].
\end{align}
 Now we use \eqref{eq:estim5.displ.conv1} back in \eqref{eq:estim1.disp.conv1} to derive that there is $C>0$ such that
\begin{align*}
	  \gamma \E\bigg[\int_0^T\|\delta Y^i_t\|^2\diff t \bigg] &\le \E\bigg[\int_0^TC_1^\eps  \e^{5\ell_B t} \int_0^t \|\delta Y^i_s\|^2\diff s   + C \frac1N\sum_{j=1}^N\|\delta X^j_t\|^2\diff t \bigg] + \varepsilon\E\bigg[ \int_0^T \frac1N\sum_{j=1}^N\|\delta Y^j_t\|^2\diff t \bigg] + C \E\big[\epsilon^N\big],
	\end{align*}
which after using Grownwall's inequality shows, in light of \eqref{eq:estim2.displ.conv} and \eqref{eq:estim5.displ.conv}, that there is $C>0$ such that
\begin{align*}
	  \E\bigg[\int_0^T\|\delta Y^i_t\|^2\diff t \bigg] &\le  C \E\big[\epsilon^N\big].
	\end{align*}
	Going back to \eqref{eq:estim3.displ.conv} now allows to conclude that $\E\big[\sup_{t\in [0,T]} \|\delta X^i_t\|^2 \big]$ satisfies the same estimate.
	This shows \eqref{eq.propagationresults}.
	The second assertion follows since by the law of large numbers, see e.g. \cite[Theorem 5.23]{kallenberg2002foundations}, $ \E[\epsilon^N] \longrightarrow 0$, as $N\longrightarrow \infty$.
\end{proof}

\begin{proof}[Proof of {\rm Theorem \ref{thm:convergence.c.mfsb.body.paper}}]

We argue in three steps. \medskip

{\bf Step 1}. We show that for $k\geq 1$ and $\varphi\in \Phi$ fixed we have that
\[
V^{N,k,\varphi}_c(\muin, \mufin) \longrightarrow V^{k,\varphi}_c(\muin, \mufin), \text{ as } N\longrightarrow \infty.
\]
Note that in light of \Cref{lemma.charact.c.k.phi}, $\hat\alpha_t = \Lambda(t, Y_t)$ is optimal for the control problem $V^{k,\varphi}_c:=V^{ k,\varphi}_c(\muin,\mufin)$ and, {\color{black} by uniqueness of $\Lambda$}, so is $\hat\balpha^{N}=(\hat\alpha^{1,N},\dots,\hat\alpha^{N,N})$ with $\hat\alpha^{i,N}_t = \Lambda(t,   Y^{i,N}_t)$ for the control problem $V^{N,k,\varphi}_c:=V^{N, k,\varphi}_c(\muin,\mufin)$, i.e., 
	\begin{align*}
			V^{N,k,\varphi}_c  = \frac1N\sum_{i=1}^N\E\bigg[\int_0^Tf(t,  X^{i,N}_t, \hat \alpha_t^{i,N},L^N( \X^{N}_t))\diff t + kg^{\varphi}(\Lc(X^{i,N}_T)) \bigg],\;
			V^{k,\varphi} _c = \E\bigg[\int_0^Tf(t,  X_t, \hat \alpha_t,\Lc( X_t))\diff t + kg^{\varphi}(\Lc(X_T)) \bigg].
	\end{align*}
{Recall now that the derivatives of $(x,a,\mu)\longmapsto f(t,x,a,\mu)$ have linear growth, see \Cref{rmk.assumptions}\ref{rmk.assumptions.1}}.	
Thus, for $( \widetilde X^i, \widetilde Y^i)$, $i=1,\dots,N$, i.i.d.~copies of $( X, Y, Z)$ solution to \eqref{eq:fbsde.MFSP.proof.c} driven by $W^i$ and $\tilde\alpha_t^i = \Lambda(t, \widetilde Y_t^i)$, there exists a linearly growing function $\beta$ such that
\begin{align}\label{convergence.step1}
\begin{split}
|V^{N,k,\varphi}_c -V^{k,\varphi}_c |& \leq \frac1N\sum_{i=1}^N\E\bigg[ \int_0^T \beta( t,X^{i,N}_t, \widetilde X_t^i , \hat \alpha^{i,N} ,  \tilde  \alpha_t^i) \big( \|  X^{i,N}_t- \widetilde X_t^i \|+  \|\hat \alpha^{i,N}_t - \tilde \alpha_t^i\| +\Wc_2 (L^N( \X^{N}_t) , \mathcal{L}(X_t))\big) \d t\bigg]\\
&\; +k \bigg( \sup_{i\leq N}  g^{\varphi}(\Lc(X^{i,N}_T)) -g^{\varphi}(\mathcal{L}(X_T))\bigg).
\end{split}
\end{align}

{We now estimate the first term on the right-hand side of \eqref{convergence.step1} which we denote $I_1$. First, by \Cref{lemma.charact.c.k.phi}\ref{lemma.charact.c.k.phi.Npart} (recall we made the convention of suppressing the superscripts $k,\varphi$) there exists $C>0$, independent of $N$, such that}
\begin{align}\label{eq:prop.con.aux2}
M:=\sup_{i\leq N} \E \Big[ \sup_{t\in [0,T]}\Big \{ \|X^{i,N}_t\|^2+ \|Y^{i,N}_t\|^2+ \| \widetilde X_t^i\|^2+\| \widetilde Y_t^i\|^2\Big \} \Big]\leq C.
\end{align}
Thus, {using the fact that $\Lambda$ is Lipschitz--continuous, see \Cref{rmk.assumptions}\ref{rmk.assumptions.3}}, and applying Cauchy--Schwarz and Jensen's inequalities, we find that for a constant $C>0$ independent of $N$ it holds that
	\begin{align*}
		I_1 \leq &\, C   \frac1N\sum_{i=1}^N \E \bigg[\int_0^T \Big( \|X^{i,N}_t - \widetilde X_t^i\|^2 + \|Y^{i,N}_t - \widetilde Y_t^i\|^2 + \frac1N\sum_{j=1}^N  \|X^{j,N}_t - \widetilde X_t^j\|^2  +  \Wc_2^2 (L^N(\widetilde \X^N_t), \Lc(X_t)) \Big) \diff t\bigg]^{\frac{1}2}\\
\leq & C  \bigg(  \frac1N\sum_{i=1}^N \E \bigg[\int_0^T \big(2  \|X^{i,N}_t - \widetilde X_t^i\|^2 + \|Y^{i,N}_t - \widetilde Y_t^i\|^2   +  \Wc_2^2 (L^N(\widetilde \X^N_t), \Lc(X_t)) \big) \diff t\bigg]\bigg)^{\frac{1}2}.
	\end{align*}

This gives the convergence of both $I_1$ and the second term in \eqref{convergence.step1} to zero as $N\longrightarrow \infty$. Indeed, under the assumptions of \Cref{thm:convergence.c.mfsb.body.paper} we have that \Cref{thm.prop.chao.c} holds, and consequently, $(X^{i,N},Y^{i,N})$ converges to $(\widetilde X^i,\widetilde Y^i)$ in $\Hc^2$, the last term converges by the law of large numbers. 

\medskip

{\bf Step 2}. We show that
\[
\limsup_{N\to \infty} V^N_c (\muin,\mufin) \leq V_c (\muin,\mufin).
\]
Let $N\geq 1$ and $\alpha$ be admissible for $V_c (\muin,\mufin)$. 
	That is $\alpha \in \mathfrak{A}$, $  X_0^\alpha\sim\muin$ and $X_T^\alpha \msim   \mufin$ for $X^\alpha$ solution to \eqref{eq:control.SDE.MF}. 
	Now, let $\tilde \alpha^{i,N}$, $\widetilde X_0^i$, $i=1,\dots,N$, be i.i.d.~copies of $\alpha$ and $X_0^\alpha$ such that $\tilde \alpha^{i,N}$ is $\F^i$-adapted and $\widetilde X_0^{i}\sim \muin$. Let $X^{i,N,\tilde \alpha^N}$ denote the solution to \eqref{eq:controlled.SDE.NP}, i.e.,
\[
X^{i,N,\tilde \alpha^N}_t = \widetilde X_0^{i} + \int_0^t\tilde \alpha_u^{i,N} + b(u, X^{i,N,\tilde \alpha^i}_u,L^N(\X^{N,\tilde \alpha^N}_u))\diff u + \sigma W^i_t,
\]
where, as before, $\X^{N,  \tilde \alpha^N }:=(X^{1,N,\tilde \alpha^N},\dots,X^{N,N,\tilde \alpha^N})$ and $ \tilde \alpha^N :=(\tilde \alpha^{1,N},\dots,\tilde \alpha^{N,N})$. Note that $ \tilde \alpha^N \in \mathcal{A}^N$. Now, let $\widetilde X^{i}$, $i=1,\dots,N$, denote the solution to 
\[
\widetilde X^{i}_t = \widetilde X_0^{i}+ \int_0^t \tilde \alpha_u^{i,N} + b(u, \widetilde X^{i}_u,\Lc(\widetilde X^{i}_u))\diff u + \sigma W^i_t,
\]
and note that by weak uniqueness we have that $\widetilde X^{i}_T\msim \mufin$.
	Moreover, by the law of large numbers $L^N\big(\widetilde \X^N\big)$ converges to $\Lc(X^{\alpha })$, as $N\longrightarrow \infty$, in the $\Wc_1$ distance.
	Consequently, the estimate 
\[
\E \bigg[\sup_{t\in[0,T ]} \big\| X^{i,N,\tilde \alpha^N }_t-\widetilde X^{i}_t\big\|\bigg] \leq C \int_0^T \E \big[   \Wc_1(L^N(\widetilde \X^{N}_r) , \mathcal{L}(X^{ \alpha}_r)) \big] \diff r,
\]
which holds for $C>0$, shows that $\Lc\big(X^{i,N,\tilde \alpha^N }\big)$ converges to $\Lc\big(\widetilde X^{i}\big)$ in the $\Wc_1$ distance for $i=1,\dots,N$. 
	Consequently, $g_c\big(\Lc(X^{i,N,\tilde \alpha^N }_T)\big)$ tends to zero as $N\longrightarrow \infty$.\medskip

Now, see \Cref{rmk.duality.Nprob}, using the fact that since $V^N_c (\muin,\mufin)<\infty$, $N\geq1$, we have that
\[
V^N_c (\muin,\mufin)=\sup_{k\geq 1}\sup_{\varphi\in \Phi}  V^{N,k,\varphi}_c (\muin,\mufin).
\]
It then follows from \eqref{eq.dualitywot.bary} and convergence of $g_c(\Lc(X^{i,N,\tilde \alpha^N }_T))$, that for any $\eps>0$, there is $N$ sufficiently large so that
\begin{align}\label{eq:prop.con.aux1}
V^N_c (\muin,\mufin)= \sup_{k\geq 1}\sup_{\varphi\in \Phi}  V^{N,k,\varphi}_c (\muin,\mufin) \leq \frac1N\sum_{i=1}^N\E\bigg[\int_0^Tf(t, X^{i,N, \tilde \alpha^N}_t, \tilde \alpha_t^{i,N},L^N(\X^{N, \tilde \alpha^N}_t))\diff t  \bigg]+ \eps.
\end{align}
Now note that as in {\bf Step 1}, we have that
\[
\frac1N\sum_{i=1}^N\E\bigg[\int_0^T\big( f(t, X^{i,N, \tilde \alpha^N}_t, \tilde \alpha_t^{i,N},L^N(\X^{N, \tilde \alpha^N}_t))-f(t, \widetilde X^{i}_t, \tilde \alpha_t^{i,N},\Lc(\widetilde X^{i}_t))\big) \diff t  \bigg]\longrightarrow 0, \text{ as } N\longrightarrow \infty,
\]
Thus, since $\tilde \alpha^{i,N}\sim\alpha \in \Ac$, by the arbitrariness of $\eps$ we find that
\[
\limsup_{N\to \infty}  V^N_c (\muin,\mufin)\leq   \E\bigg[\int_0^Tf(t, X_t, \tilde \alpha_t,\mathcal{L}(\widetilde X_t)) \diff  t   \bigg]
\]
Since $\tilde \alpha^{i,N}\sim\alpha \in \Ac$, $\widetilde X_0^{i}\sim \muin$ and $\widetilde X_T^{i}\msim \mufin$ the claim follows.\medskip

\medskip

{\bf Step 3:} We show that
\[
\liminf_{N\to \infty} V^N (\muin,\mufin) \geq V (\muin,\mufin).
\]
Recall that by weak duality we have that $V^N (\muin,\mufin)  -V^{N,k,\varphi} (\muin,\mufin)\geq 0$ for any $k\geq 1$, $\varphi\in \Phi$, and thus 
\begin{align*}
V^N(\muin,\mufin)  - V(\muin,\mufin)   &=  V^N(\muin,\mufin)   -V^{N,k,\varphi}(\muin,\mufin)+ V^{N,k,\varphi}(\muin,\mufin)  - V(\muin,\mufin) \\
& \geq   V^{N,k,\varphi}(\muin,\mufin)   -V^{k,\varphi}(\muin,\mufin)  +V^{k,\varphi}(\muin,\mufin) - V^{k}(\muin,\mufin)  +V^{k}(\muin,\mufin)  - V(\muin,\mufin) .
\end{align*}
The claim follows since the first term on the right goes to zero by {\bf Step 1}, the second term goes to zero, alongside a maximizing sequence $(\varphi^m)_{m\geq 1}\subseteq \Phi$, by \eqref{eq.supinf.fixedk} and the third goes to zero by \eqref{eq.limitink}. {\bf Step 2} and
{\bf Step 3} establish the result.
\end{proof}

\begin{remark}\label{rmk.convergenceN.mainproblem}
{Let us now elaborate on the statement given in {\rm \Cref{rmk:convergence}}. 
	As it is clear from the proof of {\rm \Cref{thm:convergence.c.mfsb.body.paper}}, the key argument is Step 1, which is a consequence of {\rm\Cref{thm.prop.chao.c}}.
	It turns out that if $\mufin$ admits a displacement convex penalty function $g$, one can obtain {\rm\Cref{thm.prop.chao.c}} for the {\rm FBSDE}s associated with $V^{k}(\muin,\mufin)$ and $V^{N,k}(\muin,\mufin)$.
	We stress these are the penalized versions of $V(\muin,\mufin)$ and $V^{N}(\muin,\mufin)$ given by \eqref{eq:problem} and \eqref{eq:N.problem}, respectively.
	Indeed, in this case, the same argument in {\rm\Cref{thm.prop.chao.c}} is carried out for the {\rm FBSDE}s \eqref{eq:fbsde.reg.MFSP} and $(X^{i,N,k}, Y^{i,N,k},Z^{i,j,N,k})$, $i=1,\dots,N$, given by \eqref{eq:N.fbsde.proof.main.c} with $Y_T^{i,N,k}=k \partial_\mu g(\Lc (X^{i,N,k}_T)(X_T^{i,N,k})$.
	The proof follows as the corresponding inequality in \eqref{eq.disp.conv.proof} holds since $g$ is displacement convex.}
\end{remark}

\begin{remark}\label{rmk.convergenceN.regularity}
Let us now further discuss the unviability of two approaches based on other known scenarii in which the study of the convergence problem in mean field games (and uniqueness of mean field games) is often achieved without the need to impose convexity assumptions. Indeed, in the literature, either small time conditions or monotonicity/convexity conditions are always needed to achieve strong convergence results {\rm \cite{cardaliaguet2015master,jackson2023quantitative,lauriere2021convergence}}.
In our setting, for $k$ fixed, the {\rm FBSDE} associated with $V^{k}(\muin,\mufin)$ is also well-posed over a small time horizon, said $T_o^k$. This suggests two frameworks for the study of the convergence result.
\begin{enumerate}[label=$(\roman*)$, ref=.$(\roman*)$,wide,  labelindent=0pt]
\item A first approach would be based on the results of {\rm \cite{chassagneux2014probabilistic}} which provide conditions for the well-posedness of bounded solutions to the so-called master equation in two scenarii. The first one is a small time horizon $T_o^k$, whereas the second scenario iterates the previous one to be able to obtain a result over arbitrary time horizons. Though this second result seems viable at first sight, the conditions under which the iteration can be carried up impose assumptions on the type of interaction between the particles, which are more restrictive than those in {\rm \Cref{thm:convergence.mfsb}}, it also requires extra regularity assumptions on the data and, perhaps more critically, it imposes the same convexity assumptions on the penalty function prescribed by {\rm \Cref{def.dispconv}}.
\item Alternatively, the theory of Malliavin calculus allows us to identify another scenario under which one can show the existence of $T_o^k <\infty$, such that for $T\leq T_o^k$, there is a bounded solution $(Y^k,Z^{k})$ to \eqref{eq:fbsde.MFSP.proof.c}.\label{rmk.convergenceN.regularity.ii}
\end{enumerate}
In either setting, the particle system associated with $V^{N,k}(\muin,\mufin)$ converges to that of $V^{k}(\muin,\mufin)$ over $T_o^k$, for every $k\geq 1$. However, to establish $V^N(\muin,\mufin)\longrightarrow V (\muin,\mufin)$ we see from {\bf Step 3} in the proof of {\rm Theorem \ref{thm:convergence.c.mfsb.body.paper}} that there needs to exist $T_o$ under which the convergence of the associated {\rm FBSDEs} holds uniformly in $k$. Unfortunately, in either $(i)$ or $(ii)$ above $T^k_o$ decreased to $0$ as we let $k$ go to infity.
\end{remark}

\subsection{The at most quadratic cost case: Proof of Theorem {\rm \ref{thm:convergence.mfsb}}}\label{sec.proof.entropic.conv}

The result follows from specializing \Cref{thm:convergence.c.mfsb.body.paper} to the case in which $f$ and $b$ are as in \Cref{thm:existence.mfsb} with the strengthening that $a\longmapsto f_1(t,a)$ is uniformly strongly convex.
	In this case, the Hamiltonian \eqref{eq:hamiltonial.particle.system} and the functionals $F$ and $B$ takes the form
\begin{align*}
	H^N_1(t, {\bf y})&=\sup_{a^1,\dots,a^N}\bigg\{ \frac1{N}\sum_{i=1}^N    \frac 12 \|a^i\|^2   + \sum_{i=1}^N   a^i\cdot y^i\bigg\} ,\;
	H^N_2(t,{\bf x}, {\bf y},L^N({\bf x}))= -\frac1N\sum_{i=1}^N  \sum_{j=1}^N ( x^i - x^j)  \cdot y^{i} ,\\	 
		 F(t,x,y,L^N({\bf x},{\bf y}))&=- y + \frac1N\sum_{j=1}^N  y^{j},\; B(t,x,y,L^N({\bf x}))=-y-\frac1N\sum_{j=1}^N ( x - x^j).
\end{align*}

 
Let us argue the result. We will use \Cref{thm:convergence.c.mfsb.body.paper} for which we must verify {\rm Assumptions \ref{ass.b.f}} and \ref{ass.F}. 
	By assumption $f_1$ is strongly convex and $f_2\equiv 0$, so, as in the proof of {\rm \Cref{thm:existence.mfsb}}, \Cref{ass.b.f} and \Cref{ass.F} hold in light of the choice of $\Psi$.
	The result follows from \Cref{thm:convergence.c.mfsb.body.paper}. 
	We remark that in the entropic case, $\Lambda(t,y)=-y$ and \eqref{eq:displaassumpLambda} holds with equality and $\lambda_{f_1}=1$\hfill \qed

\begin{appendix}
\section{Appendix}\label{appendix}


We collect here the belated proofs of some technical results.\medskip

The next result states that the regularity conditions on $\muin$, $\mufin$ and $b$ in \cite{backhoff2020mean} are sufficient for the feasibility assumption, i.e., $V(\muin,\mufin)<\infty$, for cost functions beyond the quadratic cost.

\begin{lemma}\label{lemma.feasibility}
Let $A=\R^\xdim$, $b(t, x, \mu) = \int_{\R^m}\nabla\Psi(x-\tilde x)\mu(\d \tilde x)$ for $\Psi:\R^\xdim\longrightarrow \R$ symmetric, twice continuously differentiable and satisfying $\sup\{v\cdot\nabla^2 W(z)\cdot v: z,v\in \R^\xdim, \|v\|=1\}<\infty$, and $f:[0,T]\times \RR^\xdim \times \R^d \times \cP_2(\RR^\xdim)\longrightarrow \RR$ be measurable and of at most quadratic growth, i.e., \eqref{eq.qgrowth} holds. Suppose that $\muin,\mufin\in \Pc_2(\R^\xdim)$ and $\tilde \Fc(\muin),\tilde\Fc(\mufin)<\infty$, where
\[
\Pc_2(\R^\xdim)\ni  \mu\longmapsto \tilde \Fc(\mu)=\begin{cases}
\int_{\R^\xdim} \log (\mu(x))\mu(\d x)+ \int_{\R^\xdim} b(t,x,\mu) \mu(\d x), &\text{ if } \mu\ll \lambda \\
\infty, & \text{\rm otherwise.}
\end{cases}
\]
Then, $V(\muin,\mufin)<\infty$.
\begin{proof}
Under the above assumptions, it follows from \cite[Proposition 1.1]{backhoff2020mean} that $\Vc_{\rm e}(\muin,\mufin)<\infty$. Consequently, by \cite[Lemma 1.1]{backhoff2020mean}, there is an $A$-valued, progressively measurable process $\alpha$ such that
\[
\E^\P\bigg[\int_0^T\|\alpha_t\|^2\d t\bigg]<\infty,\, X_0^\alpha \sim \muin,\text{ and, } X_T^\alpha \sim \mufin,
\]
for $X^\alpha$ solution to \eqref{eq:control.SDE.MF}. It then follows from standard estimates that $\E\big[ \sup_{t\in [0,T]} \|X_t\|^2\big]<\infty$ and thus $V(\muin,\mufin)<\infty$.
\end{proof}
\end{lemma}

For the reader's convenience, we present here the statement of \citeauthor*{backhoff2020nonexponential} \cite[Lemma A.1]{backhoff2020nonexponential}, which was used throughout the document. For the sake of completeness, we mention that $f_1$ as in {\rm \Cref{ass.b.f}} satisfied the coercivity condition \cite[${\rm (TI)}$]{backhoff2020nonexponential}, namely, $f_1(t, a)/\|a\|\longrightarrow \infty,$ as $\|a\|\longrightarrow \infty$.

\begin{lemma}\label{lem:BaLaTa}
Let $f_1$ be as in {\rm \Cref{ass.b.f}}. Suppose $(\alpha^n)_{n\geq 1}$ is a sequence of $\L^1([0,T])$-valued random variables possibly defined on different probability spaces. Let $A_n(t):=\int_0^t \alpha^n_s \d s$ and suppose there exists $C>0$ such that, for each $n$,
\[
\E\bigg[ \int_0^t f_1(s,\alpha_s^n)\d s\bigg]\leq C.
\]
Then there exist a continuous process $A$, a subsequence $A_{n_k}$ which converges in law in $\Cc([0,T],\R^m)$ to $A$, and $\alpha$ such that
\[
\liminf_{k\to \infty} \E\bigg[\int_0^t f_1(s,\alpha_s^{n_k})\d s\bigg]\geq \E\bigg[\int_0^t f_1(s,\alpha_s)\d s\bigg]
\]
and $A_t=\int_0^t \alpha_s\d s$. In particular, $(A_n)_{n\geq 1}$ is tight.
\end{lemma}

\begin{lemma}\label{lem:weakLpprojection}
Let $p>1$ and $(X^n)_{n\geq 1} \subseteq \L^p([0,T]\times\Omega)$ bounded. 
	Then, up to a subsequence, $(X^{n})_{n\geq 1}$ converges weakly in $\L^p([0,T]\times\Omega)$ and there is $D\subseteq [0,T]$, ${\rm Leb}(D)=T$, such that for every $t\in D$, $(X^{n}_t)_{n\geq 1}$ converges weakly in $\L^1(\Omega)$.
	In particular, $D$ is dense.
\begin{proof}
We first note that the existence of a weakly convergence subsequence follows immediately from the boundedness of $(X^n)_{n\geq 1}$ in $\L^p([0,T]\times\Omega)$, $1<p<\infty$, see \citeauthor*{ambrosio2000functions} \cite[Theorem 1.36]{ambrosio2000functions}. Thus, there is $X\in \L^p([0,T]\times\Omega)$ such that, using again $(X^{n})_{n\geq 1}$ to denote the subsequence, $(X^{n})_{n\geq 1}$ converges to $X$ weakly in $\L^p([0,T]\times\Omega)$. We also note that the last statement follows from the second since ${\rm Leb}(D)=T$ implies $D$ is dense in $[0,T]$. \medskip

Let us now argue the second statement. By Mazur's lemma, there exists $(\gamma^{k}_n)_{n\leq k \leq N(n)}$, $n\geq 1$, $\sum_{k=n}^{N(n)} \gamma^k_n=1$, for all $n\geq 1$, $\gamma^{k}_n\geq 0$, such that $\tilde X^n:=\sum_{k=n}^{N(n)} \gamma^k_n  X^k$ converges to $ X$ in $\L^p([0,T]\times\Omega)$. We claim that $\tilde X^n$ satisfies the second statement. That is there is $D\subseteq [0,T]$, ${\rm Leb}(D)=T$, such that for every $t\in D$, $(\tilde X^{n}_t)_{n\geq 1}$ converges weakly in $\L^1(\Omega)$. To show the claim, let $t\in [0,T]$, $K\in \L^\infty(\Omega)$, $\delta>0$ and note that
	\begin{align}\label{lem:weakLpprojection.1}
	\E\big[K(\tilde X^{n}_t-X_t)\big]=  \E\bigg[  \frac{1}\delta \int_{t}^{t+\delta } K (\tilde X^{n}_t-\tilde X^{n}_s) \diff s\bigg]  + \E\bigg[  \frac{1}\delta \int_{t}^{t+\delta } K (\tilde X^{n}_s-X_s) \diff s\bigg]  + \E\bigg[  \frac{1}\delta \int_{t}^{t+\delta } K (X_s-X_t) \diff s\bigg]  .
	\end{align}
	
We now notice that
\[
 \frac{1}\delta \int_{t}^{t+\delta } K (\tilde X^{n}_s-X_s) \diff s  \leq   \sup_{\delta>0} \frac{1}\delta \int_{t}^{t+\delta }  \|K(\tilde X^{n}_s-X_s)\|  \diff s=  \Mc [K(\tilde X^n-X)](t), 
\]
where for any integrable function $f:[0,T]\longrightarrow \R^n$, $\Mc[f](t)$ denotes the Hardy–Littlewood maximal operator, given by
\[
\Mc[f](t):= \sup_{\delta>0} \frac{1}\delta \int_{t}^{t+\delta }  \|f(s)\|  \diff s.
\]
By the strong Hardy–Littlewood maximal inequality, see \citeauthor{rudin1987real} \cite[Theorem 8.18 and Equation $(6)$]{rudin1987real}, there exists a constant $C>0$, depending only on $p$ (the dimension of $[0,T]$ is 1), such that 
\[
\E\big[ \|  \Mc [K(\tilde X^n-X)]\|_{\L^p([0,T])}^p\big] \leq C^p\E\big[ \|K(\tilde X^n-X)\|_{\L^p([0,T])}^p\big]=C^p\|K(\tilde X^n-X)\|_{\L^p([0,T]\times\Omega)}^p,
\]
 Thus, $\Mc [K(\tilde X^n-X)]$ converges to zero in $\L^p([0,T]\times \Omega)$. In particular, it converges weakly in $\L^1([0,T]\times \Omega)$, meaning that
\[
\int_0^T\E\big[ \Mc [K(\tilde X^n-X)](t)\big] \diff t \longrightarrow 0,\text{ as, }n\longrightarrow\infty.
\]
Since $\Mc [K(\tilde X^n-X)]$ is non-negative, we deduce that, up to a subsequence, $\E [ \Mc [K(\tilde X^n-X)](t) ]$ converges to zero as $n\longrightarrow\infty$ for almost every $t\in [0,T]$. Back in \eqref{lem:weakLpprojection.1}, we have that

\[
\E\big[K(\tilde X^{n}_t-X_t)\big]\leq  \E\bigg[  \frac{1}\delta \int_{t}^{t+\delta } K (\tilde X^{n}_t-\tilde X^{n}_s) \diff s\bigg]  + \E\big[  \Mc [K(\tilde X^n-X)](t)\big] + \E\bigg[  \frac{1}\delta \int_{t}^{t+\delta } K (X_s-X_t) \diff s\bigg] .
\]
We now let $\delta\longrightarrow 0$ first, so that the first and third terms go to zero for almost every $t\in [0,T]$ by Lebesgue differentiation theorem, see \cite[Theorem 7.10]{rudin1987real}, and then let $n\longrightarrow \infty$. This proves the claim.\medskip

With this, we notice that
\[
\E[ K( \tilde X^n_t-X_t)]=\sum_{k=n}^{N(n)}\gamma^k_n\E[ K(X_t^k-X_t)]\longrightarrow 0, \text{ as } n\longrightarrow \infty.
\]
so, up to a subsequence, \( \E[ K(X_t^n-X_t)]\longrightarrow 0, \text{ as } n\longrightarrow \infty\) as desired.
\end{proof}
\end{lemma}

\section*{Proofs of the finite particle approximation}\label{appendix:characterization.Nparticle}
\paragraph*{Proof of \Cref{lemma.conv.existence}.}
Let us first notice that by weak duality we have that
\begin{gather}
V_c(\muin, \mufin)\geq \sup_{k\geq 1} V_c^k(\muin, \mufin),\; \text{and, }
V_c^k(\muin, \mufin)\geq \sup_{\varphi\in \Phi} V_c^{k,\varphi}(\muin, \mufin),
\label{eq.weakduality.kestimate}
\end{gather}
and, consequently,
\begin{align}
V_c(\muin, \mufin)\geq \sup_{k\geq 1}  \sup_{\varphi\in  \Phi} \inf_{\alpha \in \mathfrak{A}} J(\alpha,k,\varphi) .\label{eq.weakduality.k.phi.estimate}
\end{align}

We argue $(i)$. The proof follows that of \Cref{eq:lemma.k} with the difference that in \eqref{eq.lemmak.continuity}, i.e.,
	 \begin{align*}
	 	g_c(\Lc(\hat X^{\alpha^k}_T)) \leq \frac1k\Big(V_c(\muin,\mufin)- C \Big)+\frac{1}k,
	\end{align*}
	we exploit the fact that $\mu\longmapsto g_c(\mu)$ is lower semicontinuous with respect to the weak topology, see \cite[Theorem 2.9]{backhoff2019existence}, to deduce that $\Lc(X_T^{\hat \alpha})\msim \mufin$.\medskip

Let us argue $(ii)$. Note that \eqref{eq.supsupinf} follows from \eqref{eq.limitink} and \eqref{eq.supinf.fixedk} since $(V^{k}_c(\muin,\mufin))_{k\geq1}$ is increasing. 
	Let $k\geq 1$ and $\varphi\in  \Phi$ be fixed and note that there is a sequence $(\alpha^{k,\varphi,n})_{n\geq 1}$ such that
\begin{align}\label{eq.infsup.estimate}
V_c^{k,\varphi}(\muin, \mufin)\geq \E\bigg[\int_0^T f(t, X^{\alpha^{k,\varphi,n}}_t, \alpha_t^{k,\varphi,n},\mathcal{L}(X^{\alpha^{k,\varphi,n}}_t))\diff t \bigg]		+k g^\varphi(\Lc(X_T^{\alpha^{k,\varphi,n}}))  -\frac1n.
\end{align}

Notice that as $\varphi$ is bounded from below, say by $\ell\in \R$, $k g^\varphi(\Lc(X_1^{\alpha^{k,\varphi,n}}))=k\Lc(X_T^{\alpha^{k,\varphi,n}})(\varphi)- k \mufin(\varphi)\geq k\ell - k \mufin(\varphi)$, for all $n\geq 1$.
	Thus, since $V_c^k(\muin, \mufin) \leq V_c(\muin, \mufin) <\infty$ and $f_2$ is bounded from below, there exists a constant $C^{k,\varphi}>0$, independent on $n$, such that
\[
 \E\bigg[\int_0^1 f_1(t, \alpha_t^{k,\varphi,n})\diff t \bigg]\leq C^{k,\varphi}.
\]

Thus, as in \Cref{lemma:infty.problem-k.penalized}, we deduce the existence of $\alpha^{k,\varphi}\in \mathfrak{A}$ and the convergence in law of the associated controlled processes $X^{k,\varphi,n}:=X^{\alpha^{k,\varphi,n}}$ to $X^{k,\varphi}:=X^{\alpha^{k,\varphi}}$, as $n\longrightarrow \infty$, where $X^{k,\varphi}$ is uniquely given by the controlled SDE \eqref{eq:control.SDE.MF} with $\alpha^{k,\varphi}$.
Letting $n\longrightarrow \infty$ in \eqref{eq.infsup.estimate}, we deduce $\alpha^{k,\varphi}$ is optimal for $V_c^{k,\varphi}(\muin, \mufin)$, since  
\begin{align*}
V_c^{k,\varphi}(\muin, \mufin)\geq J(\alpha^{k,\varphi},k,\varphi).
\end{align*}
Thus, we have from \eqref{eq.weakduality.kestimate} that
\[
\pushQED{\qed}
V_c^{k}(\muin, \mufin)\geq \sup_{\varphi\in  \Phi} V_c^{k,\varphi}(\muin, \mufin) = \sup_{\varphi\in  \Phi}  J(\alpha^{k,\varphi},k,\varphi)\geq \inf_{\alpha
\in \mathfrak{A}} \sup_{\varphi\in  \Phi} J(\alpha,k,\varphi)=V_c^{k}(\muin, \mufin). \qedhere
\popQED
\]

\paragraph*{Proof of \Cref{lemma.charact.c.k.phi}.}
The proof of $(i)$ is analogous to that of \Cref{lemma:infty.problem-k.penalized}. Let us argue $(ii)$. Let $N,k\geq1$ and $\varphi\in \Phi$ be fixed.
	For every $n \in \mathbb{N}$, there is $\balpha^{n}:=(\alpha^{1,n},\dots,\alpha^{N,n}) \in \Ac^N$ such that
	\begin{equation}
	\label{eq:lower.bound.VinfknN}
		V^{N,k,\varphi}_c(\muin, \mufin) \ge \frac1{N}\sum_{i=1}^N \EE\bigg[\int_0^Tf(t, X^{i,\balpha^{n}}_t, \alpha^{i,n}_t,L^N(\X^{N,\balpha^n}_t))\diff t + kg^{\varphi}(\cL(X^{i,\alpha^{n}}_T)) \bigg] - \frac1n.
	\end{equation}
	Since $V^{N,k,\varphi}(\muin,\mufin)\le V^N(\muin, \mufin)<\infty$ and the functions $f_2$ and $g^\varphi$ are bounded from below, there is $C^{N,k,\varphi}>0$ independent of $n$, such that for $f_1^N(t,{\bf a}):=\frac1N \sum f_1(t,a^i)$, we have that
	\begin{equation}
	\label{eq:bound.f1N}
  \EE\bigg[\int_0^Tf_1^N(t, \alpha^{n}_t)\diff t\bigg] \le C,\; i=1,\dots,N.
	\end{equation}
	Since $f_1^N(t,{\bf a})$ is convex, by \Cref{lem:BaLaTa},
	the sequence $(\int_0^t\alpha^{n}_s\diff s)_{n\geq 1}$ is tight and converges to $\int_0^t\hat \alpha_s^{N}\diff s$ in law in $C([0,T],(\RR^m)^N)$ and 
	it holds that
	\begin{equation*}
	  	\liminf_{n\to \infty}\EE\bigg[\int_0^Tf_1^N(s, \alpha^{n}_s)\diff s\bigg]\ge \EE\bigg[\int_0^Tf_1^N(s, \hat \alpha_s^N)\diff s\bigg].
	\end{equation*} 
	Note that this last inequality shows $\hat\alpha^N=(\hat \alpha^{1,N},\dots,\hat\alpha^{N,N})\in \Ac^N$.\medskip
	
	Let us now argue the weak convergence of the associated controlled processes. We claim that $(\X^n)_{n\geq 1}$, where $\X^n:=\X^{N,\alpha^{n}}$, is tight.

	 For this note that there is $p>1$
	\begin{align*}
	\|X_t^{i,n}\|^p& \leq  C\bigg(\|X_0^i\|^p+ \int_0^T \| \alpha_r^{i,n}\|^p\diff r + \int_0^t \| b(r,X_r^{i,n},L^N(\X_r^n))\|^p \diff r+ \|\sigma W_t^i\|^p\bigg)\\
	& \leq   C\bigg(1+ \|X_0^i\|^p+ \int_0^T f_1(r, \alpha_r^{i,n}) \diff r + \int_0^t  \Big( \|X_r^{i,n}\|^p +\frac1N\sum_{j=1}^N \|X_r^{j,n}\| ^p\Big) \diff r+ \|\sigma W_t^i\|^p\bigg),
	\end{align*}
	{\color{black} where we use $f_1(t,a)\geq C_1+C_2 \|a\|^p$}. Averaging over $i=1,\dots,N,$ and using Gronwall's inequality, we find thanks to \eqref{eq:bound.f1N} that there is $C>0$, independent of $n$, such that
	$
	\E\big[ \frac1N\sum_{i=1}^N \sup_{t\in [0,T]} \|X_t^{i,n}\|^p\big] \leq  C
	$.
	Which in turn implies back in the previous estimate that there is $C>0$ such that
	\begin{align}\label{eq:boundXXntight}
	\sup_{n\geq 1} \E\bigg[\sup_{t\in [0,T]} \|\X_t^{n}\|^p\bigg]\leq C.
	\end{align}
	
	Similarly, for the same $p>1$, we may find that there is $C>0$
	\begin{align*}
	\|X_t^{i,n}-X_s^{i,n}\|^p \leq  C\bigg( |t-s|^{p-1}  \int_0^T f_1(r, \alpha_r^{i,n})  +\|X_r^{i,n}\|^p+\frac1N\sum_{j=1}^N \|X_r^{j,n}\|^p \diff r+ |t-s|^{\frac p2}\|\sigma\|^p\bigg),
	\end{align*}
	which, together with \eqref{eq:bound.f1N}, \eqref{eq:boundXXntight} and Aldous criterion, see \cite[Theorem 16.11]{kallenberg2002foundations}, allows us to conclude that $(\X^n)_{n\geq 1}$ is tight.\medskip

	By Skorokhod's extension theorem, we can find a common probability space $(\mybar \Omega, \mybar \cF, \mybar \P)$ supporting $(\mybar \X^n,\int_0^\cdot \mybar \alpha_t^{n} \diff t, \mybar \W^n) \stackrel{d}{=} (\X^n,\int_0^\cdot \alpha_t^{n} \diff t, \W)$, for all $n\geq 1$, and $(\mybar \X,\int_0^\cdot \mybar \alpha_t^{N} \diff t,\mybar \W)\stackrel{d}{=} (X,\int_0^\cdot \hat \alpha_t^{N} \diff t, \W)$, and on which $(\mybar \X^{n} ,  \int_0^\cdot \mybar\alpha^{n}_t \diff t  , \mybar \W^n)\longrightarrow ( \mybar \X , \int_0^\cdot \mybar{\alpha}_t^N\diff t , \mybar \W)$, $\mybar \P\text{--a.s.}$
	We may now use the Lipschitz continuity of $b$ to derive that for any $i=1,\dots,N$, it holds that
	  \begin{equation*}
	 	  \|\overline X^{i,n}_t - \overline X^{i,N}_t \|  \le   \|\overline X^{i,n}_0 - \overline X^{i,N}_0 \| +  \bigg\|\int_0^t\alpha^{i,n}_s  - \hat \alpha_s^{i,N} \diff s\bigg\|  +  \int_0^t\ell_b    \|X^{i,\alpha^{n}}_s - X^{i,\hat \alpha^N}_s \| +  \frac{\ell_b}N \sum_{j=1}^N  \E\big[ \|X^{j,\alpha^{n}}_s - X^{j,\hat \alpha^N}_s \|\big] \d s+ \|\overline W^{i,n}_t - \overline W^i_t \|
	 \end{equation*}
	 so that summing over $i$ and using Gronwall's inequality, we arrive at
	 \begin{equation*}
	 	 \mybar \EE \bigg[\sup_{t\in [0,T]} \|\mybar \X^{n}_t - \mybar \X^{N,\bar \alpha^N}_t \|\bigg] \le \mathrm{e}^{2 \ell_bT}   \mybar \EE \bigg[  \|\overline \X^{n}_0 - \overline \X^{N}_0 \|  + \sup_{t\in [0,T]} \bigg\|\int_0^t\alpha^{n}_s\diff s - \int_0^t\hat \alpha_s^N \diff s\bigg\| +  \sup_{t\in [0,T]}  \|\overline \W^{n}_t  - \overline \W_t \| \bigg].
	 \end{equation*}
	 Consequently, $\Wc_1\big(\Lc\big(\mybar \X^{n}\big),\Lc\big(\mybar \X^{N, \bar \alpha^N}\big)\big)=\Wc_1\big(\Lc\big( \X^{n}\big),\Lc\big( \X^{ \hat \alpha^{N}}\big)\big)$ converges to zero as $n$ goes to infinity. We conclude that $(\X^{N,\alpha^{n}})_{n\ge0}$ converges in law to $\X^{\hat \alpha^N}$, and for $i=1,\dots, N$, $X^{i,\hat \alpha^N}$ is uniquely given by the controlled SDE
	 \begin{equation*}
	 	\diff X^{i,\hat \alpha^N}_t = \hat \alpha_t^{i,N} + b(t, X^{i,\hat \alpha^N}_t, L^N(\X^{N,\hat\alpha}_t))\diff t + \sigma\diff W_t,\quad X^{i,\hat \alpha^N}_0 \sim \muin.
	 \end{equation*}
	 Therefore, taking the limit in $n$ in \eqref{eq:lower.bound.VinfknN}, it follows by continuity of $f$ and $g$ that
	 \begin{equation*}
	 	V^{N, k}(\muin, \mufin) \ge \frac1{N}\sum_{i=1}^N \EE\bigg[\int_0^Tf(t, X^{i,\hat \alpha^N}_t, \hat \alpha^{i}_t,L^N(\X^{N,\hat \balpha}_t))\diff t + kg^\varphi(\cL(X^{i,\hat \alpha^{i}}_T)) \bigg]  ,
	 \end{equation*}
	showing that $\hat \alpha^N$ is optimal since it is admissible.
	The characterization now follows from the maximum principle. Indeed, for Problem \eqref{eq:N.problem.c-k.penalized} the Hamiltonians and terminal condition take the form
	\begin{align}\label{eq:hamiltonial.particle.system}
	\begin{split}
	H^N_1(t, {\bf y})&:=\sup_{a^1,\dots,a^N}\bigg\{ \frac1{N}\sum_{i=1}^N f_1(t,a^i) +\sum_{i=1}^N  a^i\cdot y^i\bigg\} ,\, g^{\varphi,N}(\mu):=\frac1{N}\sum_{i=1}^N k g^{\varphi}(\pi^i_\# \mu) ,\\
	H^N_2(t,{\bf x}, {\bf y},L^N({\bf x}))&:=\frac1{N}\sum_{i=1}^N f_2(t,x^i,L^N({\bf x})) +\sum_{i=1}^N   b(t,x^i,L^N({\bf x}))\cdot y^i,
	\end{split}
	\end{align}

	where ${\bf x}=(x^1,\dots, x^N)\in (\R^{\xdim})^N$, ${\bf a}=(a^1,\dots, a^N)\in A^N$, ${\bf y}=(y^1,\dots, y^N)\in(\R^{\xdim})^N$, and $\mu \in \Pc_2((\R^{\xdim})^N)$.
	Thus, the maximum principle leads to the system which for $X^{i,N}:=X^{i,\hat \alpha^N}$, $\tilde Y^{i,N}$, and, $\tilde Z^{i,N}$, $i=1,\dots,N$, satisfies 	\[
	\begin{cases}
		\diff X_t^{i,N}  = B\big(t, X_t^{i,N}, N\tilde Y^{i,N}_t, L^N(\X^{N}_t)\big)\diff t + \sigma \diff W^i_t\\
	\d \tilde Y_t^{i,N} =  - \frac{1}N F\big ( t,X_t^{i,N},N \tilde Y_t^{i,N} , L^N(\X_t^{N},N \tilde \Y_t^{N})\big)\d t + \sum_{k=1}^N \tilde Z^{k,i} \d W^k_t\\
	X_0^{i,N} \sim \muin,\;  \tilde Y_T^i=\frac{k}N  \nabla g(\Lc(X_T^{i,N}))(X_T^{i,N}).
	\end{cases}
	\]
	{We also notice that \eqref{eq:hamiltonial.particle.system} and the optimality of $\hat\alpha^N$ implies that $f_1(t,\hat \alpha^{i,N}_t)+\hat \alpha^{i,N}_t\cdot (N \tilde Y^i_t)= H_1(t,N \tilde Y^i_t)$, $\d t\otimes\d \P\text{--a.e.}$, $i=1,\dots, N$, for $H^1$ given by \eqref{eq:def.H}. 
	That is, $\hat \alpha^{i,N}_t=\Lambda(t,N \tilde Y^i_t)$, for $i=1,\dots, N$}.
	The result follows from the change of variables $(\Y,\Z):=(N\tilde \Y,N\tilde \Z)$.
	We also remark that since $(\Y,\Z)$ are uniquely defined by the solution to the BSDE in \eqref{eq:N.fbsde.proof.main.c} we have that $(Y^{i,N},Z^{i,N}) \in\S^2 \times\H^2$. In particular, by classic estimates on BSDEs, see \citeauthor*{zhang2017backward} \cite[Theorem 4.2.1]{zhang2017backward}, there is $C>0$, independent on $N$, such that $\| Y^{i,N}\|_{\S^2}+\| Z^{i,N}\|_{\H^2}<C$.\qed

\end{appendix}

{\small
\bibliography{bibliography}
}

\end{document}